\documentclass{article}

\PassOptionsToPackage{numbers,square}{natbib}
\usepackage[final]{neurips_2022}

\usepackage{tabularx, enumerate,soul}
\usepackage[dvipsnames]{xcolor}
\usepackage{amsmath,amssymb,dsfont,subcaption,bigints}
\usepackage{mathtools, graphicx,color,dsfont,amsthm,bbm}
\usepackage{lipsum,bm,comment,booktabs}
\usepackage[numbers]{natbib}
\usepackage{longtable}
\usepackage{array}
\usepackage{caption}
\usepackage[colorlinks=true,linkcolor=blue,anchorcolor=blue,citecolor=blue,filecolor=black,menucolor=black,runcolor=black,urlcolor=black]{hyperref}
\usepackage{cleveref}

\usepackage{algorithm}% http://ctan.org/pkg/algorithms
\usepackage{algpseudocode}% http://ctan.org/pkg/algorithmicx
\algnewcommand\algorithmicinput{\textbf{INPUT:}}
\algnewcommand\INPUT{\item[\algorithmicinput]}
\algnewcommand\algorithmicoutput{\textbf{OUTPUT:}}
\algnewcommand\OUTPUT{\item[\algorithmicoutput]}

\newcommand{\RN}[1]{%
  (\textup{\uppercase\expandafter{\romannumeral#1}})%
}
\newcolumntype{H}{>{\setbox0=\hbox\bgroup}c<{\egroup}@{}}
\newcommand{\E}{\mathbb{E}}
\newcommand{\mP}{\mathbb{P}}
\theoremstyle{plain}
\newtheorem{thm}{Theorem}
\newtheorem{lemma}[thm]{Lemma}
\newtheorem{asp}{Assumption}

\newtheorem{defn}{Definition}

\theoremstyle{definition}

\allowdisplaybreaks

\usepackage[utf8]{inputenc} % allow utf-8 input
\usepackage[T1]{fontenc}    % use 8-bit T1 fonts
\usepackage{hyperref}       % hyperlinks
\usepackage{url}            % simple URL typesetting
\usepackage{booktabs}       % professional-quality tables
\usepackage{amsfonts}       % blackboard math symbols
\usepackage{nicefrac}       % compact symbols for 1/2, etc.
\usepackage{microtype}      % microtypography
\usepackage{xcolor}         % colors

\title{Network change point localisation under local differential privacy}

\author{
  Mengchu Li \\
  Department of Statistics\\
  University of Warwick\\
  \texttt{mengchu.li@warwick.ac.uk} \\
  \And 
  Thomas B.~Berrett \\
  Department of Statistics \\
  University of Warwick \\
  \texttt{tom.berrett@warwick.ac.uk} \\
  \AND 
  Yi Yu \\
  Department of Statistics \\
  University of Warwick \\
  \texttt{yi.yu.2@warwick.ac.uk} \\
}

\begin{document}

\maketitle

\begin{abstract}
Network data are ubiquitous in our daily life, containing rich but often sensitive information. In this paper, we expand the current static analysis of privatised networks to a dynamic framework by considering a sequence of networks with potential change points. We investigate the fundamental limits in consistently localising change points under both node and edge privacy constraints, demonstrating interesting phase transition in terms of the signal-to-noise ratio condition, accompanied by polynomial-time algorithms. The private signal-to-noise ratio conditions quantify the costs of the privacy for change point localisation problems and exhibit a different scaling in the sparsity parameter compared to the non-private counterparts. Our algorithms are shown to be optimal under the edge LDP constraint up to log factors. Under node LDP constraint, a gap exists between our upper bound and lower bound and we leave it as an interesting open problem. 

\end{abstract}

\section{Introduction}\label{sec: introduction}

Numerous application areas and everyday life routinely generate network data, which contain valuable but often sensitive information \cite[e.g.][]{potterat2002risk, onnela2007structure}.  Understanding the underlying patterns of network data while preserving individuals' privacy is crucial in modern data analysis.  Several attempts have been made, but mostly focus on studying single snapshots of networks (a.k.a.~static networks) and/or subject to central differential privacy constraint, where a central data curator is allowed to handle raw information from all individuals \citep[e.g.][]{karwa2014private, karwa2016inference, kasiviswanathan2013analyzing, chang2021edge, mohamed2022differentially}. In this paper, we are instead concerned with understanding the dynamics of a sequence of networks (a.k.a.~dynamic networks), under local privacy constraints (LDP), where no one is allowed to handle the raw data of other individuals \citep[e.g.][]{duchi2018minimax, erlingsson2014rappor, lin2022towards, wei2020asgldp, qin2017generating}.  

Dynamic networks are usually in the form of a sequence of static networks, along a linear ordering, say time.  In the dynamic networks studies, it is vital to capture the ever changing nature.  A handy and useful way to model the changes is to assume that there exists a sequence of unknown time points, where the underlying distributions change abruptly \citep[e.g.][]{tao2019statistical, wang2021optimal}.  These unknown time points are referred to as change points.  Identifying change points helps to pinpoint important events, and more accurately estimate underlying distributions, which can be regarded as stationary between two consecutive change points.  Dynamic networks change point analysis has demonstrated its success in climatology \citep[e.g.][]{malik2012analysis}, crime science\citep[e.g.][]{bright2013evolution} and neuroscience \citep[e.g.][]{cribben2017estimating,mancho2020network}, to name but a few.

Despite the growing popularity in studying dynamic networks, we have witnessed a vacuum in estimating change points while preserving data owners' privacy.  Having said this, a line of attack has been made to analyse static network data under LDP constraints, where only data owners have access to their individual raw data \citep[e.g.][]{qin2017generating, wei2020asgldp}. The private analysis of network data is complicated by the fact that different LDP conditions are required depending on the information that one wants to protect. For example, in a relationship network, users may want to protect their edge information, i.e. whether they are connected to someone else or not. As we argue in \Cref{sec: edge LDP}, formalisation to protect such information should require minimal trust between users due to the symmetric nature of network data.  In a recommending system network, users may want to protect their entire connection portfolio, representing the purchase status of a user over a collection of products.  In a brain imaging network, a patient may even prefer protecting the entire network from adversarial inference attacks.

In view of the aforementioned state of the art, we list the contributions of this paper below.

$\bullet$ To our best knowledge, this is the first time investigating change point localisation in dynamic networks, under LDP constraints.  We consider two dynamic network models, where a sequence of sensitive networks are generated from (a) inhomogeneous Bernoulli networks (IBN) (\Cref{def-ibn}) and (b) bipartite networks with possibly dependent Bernoulli entries (\Cref{def-bibn}).   Multiple change points of the raw network distributions are allowed.  To tailor to the network models, we consider two forms of privacy requirements - edge LDP and node LDP (See \Cref{sec: privacy defn}).

$\bullet$ For dynamic IBNs under edge LDP, we show a phase transition in terms of the signal-to-noise ratio, partitioning the whole parameter space into two parts: (1) the infeasibility regime where no algorithm is expected to provide consistent change point estimators and (2) the regime where a computationally-efficient algorithm is shown to output consistent estimators.  The importance of this phase transition is twofold: (1) The transition boundary is different from its counterpart in the non-private case \citep{wang2021optimal}, quantifying the cost of preserving edge privacy in localising change points.  (2) We show that a simple non-interactive randomised response \cite{warner1965randomized} based privacy mechanism is minimax rate-optimal (up to log factors) for the purpose of change point localisation among all sequentially interactive mechanisms. 

$\bullet$ For bipartite networks under node LDP, we derive an infeasibility regime which is different from that under the edge LDP.  This fundamental difference quantifies the difference between these two different LDP constraints, and can be used to help practitioners designing data collection mechanisms.  We adopt a privacy mechanism proposed in \citep{duchi2018minimax}, together with a change point estimation routine, providing a consistent change point estimator.  Supported by a minimax lower bound result, our estimator is shown to be minimax rate optimal when the number of columns is of constant order.  When the number of columns is allowed to diverge, a gap between our lower and upper bounds exists.  This echos the well-identified challenges in the high-dimensional privacy research \cite[e.g.][]{duchi2018minimax,duchi2019lower}.  We contribute a high-dimensional network example along with in-depth discussions.   

\noindent \textbf{Notation}  For any matrix $A \in \mathbb{R}^{M\times N}$, let $A_{ij}$ be the $(i,j)$-th entry of $A$, $A_i \in \mathbb{R}^N$ be the $i$-th row of $A$, $A^{\top}$ be the transpose of $A$, $\|A\|_{\infty} = \max_{1\leq i \leq M, 1\leq j \leq N}|A_{ij}|$ and $\|A\|$ denote the operator norm of $A$.  For any matrix $B \in \mathbb{R}^{M\times N}$, let $(A,B) = \sum_{1\leq i \leq M, 1\leq j \leq N }A_{ij}B_{ij}$ and $\|A\|_{\mathrm{F}} = \sqrt{(A,A)}$ be the Frobenius norm of $A$. For any vector $v \in \mathbb{R}^p$, let $\|v\|_1, \|v\|_2, \|v\|_\infty$ be the $\ell_1$, $\ell_2$ and $\ell_\infty$ vector norms respectively. For any set $S$, let $|S|$ be the cardinality of $S$.  Let $S = S'\sqcup S''$, if $S' \cup S'' = S$ and $S' \cap S'' = \emptyset$.  Let $1\{\cdot\}$ be the indicator function only taking values in $\{0, 1\}$. For two any functions of $T$, say $f(T)$ and $g(T)$, we write $f(T) \gtrsim g(T)$ if there exists constants $C>0$ and $T_0$ such that $f(T) \geq Cg(T)$ for any $T \geq T_0$, and write $f(T) \asymp g(T)$ if $f(T) \gtrsim g(T)$ and $f(T) \lesssim g(T)$. 

\subsection{Problem setup}\label{sec-problem-setup}

We consider two parallel models of dynamic networks.  The first one is built upon an IBN model, which covers a wide range of models for undirected networks, including the Erd\H{o}s--R\'enyi random graph \citep{renyi1959random}, the stochastic block model \citep{holland1983stochastic} and the random dot product graph \citep[e.g.][]{athreya2017statistical}, among others.  

\begin{defn}[Inhomogeneous Bernoulli network, IBN] \label{def-ibn}
A network with node set $\{1,\dotsc,n\}$ is an inhomogeneous Bernoulli network if its adjacency matrix $A \in \mathbb{R}^{n \times n}$ satisfies that $A_{ij} = A_{ji} = 1\{\mbox{nodes }i, j \mbox{ are connected by an edge}\}$ and $\{A_{ij},i\leq j\}$ are independent Bernoulli random variables with $\E(A_{ij}) = \Theta_{ij}$.
\end{defn}

The second model considered is a bipartite IBN with possibly correlated entries within each row of the biadjacency matrix.  See \cite{asratian1998bipartite} for more discussions on bipartite networks.

\begin{defn}[Bipartite IBN] \label{def-bibn}
A network with node set $\{1,\dotsc, n_1 + n_2\} = V_1 \sqcup V_2$, $|V_1| = n_1$ and $|V_2| = n_2$, is a bipartite IBN if its biadjacency matrix $A \in \mathbb{R}^{n_1 \times n_2}$ satisfies the following. (1) For $i \in V_1$ and $j \in V_2$, $A_{ij} = 1\{\mbox{nodes }i, j \mbox{ are connected
}\}$.  (2) For any $i_1, i_2 \in V_1$, $i_1 \neq i_2$, $\{A_{i_1, j}, \, j = 1, \ldots, n_2\}$ and $\{A_{i_2, j}, \, j = 1, \ldots, n_2\}$ are independent.  (3) For any $i \in V_1$, $j \in V_2$, $A_{ij}$ is a Bernoulli random variable with $\E(A_{ij}) = \Theta_{ij}$.
\end{defn}

Bipartite IBNs are often used in the recommending system, where each $i \in V_1$ represents a user and each $j \in V_2$ represents a product \citep[e.g.][]{liu2009personal,ju2014personal}.  An important difference between Definitions~\ref{def-ibn} and \ref{def-bibn} is that, in \Cref{def-ibn} all entries are assumed to be independent, while in \Cref{def-bibn}, entries within the same row are allowed to be arbitrarily dependent.  Dependence in networks are common in practice, for example the control-flow graph considered in \cite{zhang2020differentially} where $V_1$ corresponds to the set of users and each node in $V_2$ corresponds to a component within some software application and the dependencies therein are due to the causality between nodes.  

The change points are defined formally in \Cref{asp:model-bi} where the magnitude of the distributional change is measured by the normalised Frobenius norm. The choice of Frobenius norm captures both dense and sparse changes in the network structure, see \cite{wang2021optimal}.

\begin{asp} \label{asp:model-bi}
Let $\{A(t)\}_{t = 1}^T \subset \{0, 1\}^{n_1 \times n_2}$ be an independent sequence of adjacency matrices of IBNs defined in \Cref{def-ibn} (in which case $n_1 = n_2 = n$) or biadjacency matrices of bipartite IBNs defined in \Cref{def-bibn}, with $\mathbb{E}\{A(t)\} = \Theta(t)$.  Assume that there exist $\{\eta_1, \dotsc,\eta_K\} \subset\{2, \ldots, T\}$, with $1 = \eta_0 < \eta_1 < \dotsc<\eta_K\leq T<\eta_{K+1} = T+1$, such that $\Theta(t) \neq \Theta(t-1)$, if and only if $t \in \{\eta_1,\dotsc,\eta_K\}$.

Let $\Delta = \min_{k = 1}^{K+1}(\eta_k-\eta_{k-1})$ be the minimal spacing and $\kappa_0 = \min_{k = 1}^K \|\Theta(\eta_k)-\Theta(\eta_{k}-1)\|_\mathrm{F}/(\sqrt{n_1n_2} \rho)$ be the minimal jump size, where $\rho = \max_{t = 1}^T \|\Theta(t)\|_\infty$ denotes the entry-wise sparsity.
\end{asp}

For both models, under privacy constraints to be discussed in \Cref{sec: privacy defn}, our goal is to construct consistent estimators $\{\widehat{\eta}_k\}_{k = 1}^{\widehat{K}}$ of $\{\eta_k\}_{k = 1}^K$.  To be specific, $\{\widehat{\eta}_k\}_{k = 1}^{\widehat{K}}$ is said to be consistent if $\Delta^{-1}\max_{k = 1}^K |\widehat{\eta}_k - \eta_k| \to 0$ and $\widehat{K} = K$ holds with probability tending to 1, as the sample size~$T$ grows unbounded.

Lastly, we note that in statistical network analysis, when allowing for entry-wise sparsity, it is usually assumed that $\rho \geq \log(n)/n$ \citep[e.g.][]{wang2021optimal} to ensure there are sufficiently many observed edges. However, We do not impose lower bounds on $\rho$ in \Cref{asp:model-bi}, since to preserve privacy, the expectations of privatised network entries are inflated by a factor of the privacy level $\alpha \in (0, 1)$.  Let $\rho'$ be the sparsity parameter of the privatised networks.  Such inflation automatically ensures that $\rho' \geq \log(n)/n$, for any $\rho \in [0, 1]$ and $n>1$ (See the proof of \Cref{prop: edge upperbound}).

\section{Network local differential privacy} \label{sec: privacy defn}

To formalise different network LDP notions, we first recall a general definition of LDP.  A private mechanism is a conditional distribution, which conditional on raw data, outputs privatised data.  For a pre-specified privacy level $\alpha \geq 0$, a random object $Z_i$ taking values in $\mathcal{Z}$ is an $\alpha$-LDP version of the raw data $X_i$, if for any raw data $x$ and $x'$, any measurable set $S \subset \mathcal{Z}$, it holds that
\begin{equation}\label{eq-LDP-def}
    \frac{Q_i(Z_i \in S | X_i = x, Z_1 = z_1,\dotsc,Z_{i-1} = z_{i-1})}{Q_i(Z_i \in S | X_i = x', Z_1 = z_1,\dotsc,Z_{i-1} = z_{i-1})} \leq e^{\alpha}.
\end{equation}
Mechanisms $Q_i$'s satisfying \eqref{eq-LDP-def} are called sequentially interactive \cite{duchi2018minimax}. A privacy mechanism is $\alpha$-LDP if all output $Z_i$'s are $\alpha$-LDP.  We focus on the regime $\alpha \in (0, 1)$, where the effect of privacy is the strongest and is often the regime of primary interest \cite[e.g.][]{duchi2013local,duchi2018minimax,berrett2021changepoint,rohde2020geometrizing}. 
  
In view of \eqref{eq-LDP-def}, the LDP constraint ensures that each individual $i$ only has access to their own raw data.  As for network data, to impose LDP, it is crucial to formalise what a unit of information includes and who are the owners of each unit of information.  In the rest of this section, we consider two cases arising from different application backgrounds.

\subsection{Edge local differential privacy in inhomogeneous Bernoulli networks}\label{sec: edge LDP}

In epidemiological studies on sexually transmitted diseases, network data are formed by edges linking sexual partners \citep[e.g.][]{rocha2010information}.  A natural choice of information unit is the existence of sexual relationship among subjects.  Due to the sensitivity of such data, one may wish to consider all parties involved to be the owners of a potential link.  Inspired by such applications, we formalise the edge LDP in \Cref{def:edge-ldp}.

 Let $a_{1:n,1:n}(1:T) = \{a_{ij}(t): 1 \leq i \leq j \leq n, 1 \leq t \leq T\}$ be the upper triangular parts of a sequence of observed adjacency matrices. We consider (sequentially) interactive mechanisms where each edge $Z_{ij}(t)$ is allowed to depend on previous private information, i.e. 
 \[
 Z_{ij}(t) | A_{ij}(t) = a, \rightarrow Z_{ij}(t) \sim Q_{ij}^{(t)}\big(\cdot | a , \rightarrow Z_{ij}(t)\big),
 \]
where the notation $\rightarrow Z_{ij}(t) = \big(Z_{1:n,1:n}(< t), Z_{1:(j-1),1:(j-1)}(t), Z_{1:(i-1),j}(t)\big)$ contains all `previous' private information. Note that, without loss of generality, we have fixed an order of interaction in above. That is at each time point $t$, the sequence of privatisation is 
\begin{equation}\label{eq:orderinteraction}
    \begin{cases}
    a_{i-1,j}(t) \rightarrow a_{i,j}(t) \;&\text{when}\;  i < j \\  a_{ij}(t) \rightarrow a_{1,j+1}(t) \;&\text{when}\; i = j.
    \end{cases}
\end{equation}

\begin{defn}[Edge $\alpha$-LDP]\label{def:edge-ldp}
 We say that the privacy mechanism $Q$ defined as 
 \[
 Q(S|a_{1:n,1:n}(1:T)) = \int_{z_{1:n,1:n}(1:T) \in S} \prod_{t= 1}^T \prod_{i = 1}^j \prod_{j = 1}^n d Q_{ij}^{(t)}(z_{ij}{(t)} | a_{ij}{(t)} , \rightarrow z_{ij}(t) )
 \]
is edge \(\alpha\)-LDP, if for any integer $1 \leq t \leq T$, any integer pair $1 \le i \le j \le n$, any measurable set $S \subset \mathcal{Z}$ and any $a_{ij}{(t)}, a_{ij}'{(t)} \in \{0, 1\}$, it holds that
\begin{equation} \label{eq:edgeldp}
    \frac{Q_{ij}^{(t)}(Z_{ij}(t) \in S | a_{ij}{(t)} , \rightarrow z_{ij}(t))}{Q_{ij}^{(t)}(Z_{ij}(t) \in S | a_{ij}'{(t)} , \rightarrow z_{ij}(t))} \leq e^{\alpha}.
\end{equation}
\end{defn}

\Cref{def:edge-ldp} allows sequential interactive mechanisms which is more general than existing edge LDP notions \citep[e.g.][]{qin2017generating,wei2020asgldp} where only non-interactive mechanisms are considered. In addition, existing definitions \citep[e.g.][]{qin2017generating,wei2020asgldp} require that for any $i \in \{1, \dotsc, n\}$ and any $x, x' \in \mathbb{R}^{n}$ with $\|x-x'\|_1 = 1$,
\begin{equation}\label{eq: edgeldpusual}
    Q_i^{(t)}(Z_i{(t)} \in S | A_{i}{(t)} = x)/Q_i^{(t)} (Z_{i}{(t)} \in S | A_{i}{(t)} = x') \le e^{\alpha}.
\end{equation}
Mechanisms satisfying \eqref{eq: edgeldpusual} requires trust between nodes. If a node does not follow the protocol correctly, or their data are intercepted, they may reveal information on other nodes in the network.  This is not the case with LDP mechanisms in other settings, where the privacy of an individual is guaranteed regardless of the behaviour of other individuals. Our definition~\eqref{eq:edgeldp} does not suffer from this since to privatise each edge between two nodes, \eqref{eq:edgeldp} implicitly requires that both parties to agree on their status and the privatised result so that the trust issue can be prevented.

\subsection{Node local differential privacy in bipartite inhomogeneous Bernoulli networks}\label{sec-node-ldp-def}

In a Netflix data set, one may model the viewing history by a dynamic bipartite IBN, where each row represents a user, each column represents a movie and each snapshot of network gathers the viewing information within a short time frame.  It is reasonable to consider an information unit to be the viewing history of a user within a time frame, which is a row in a biadjacency matrix.  Inspired by such applications, we formalise the bipartite node LDP in \Cref{def:node-ldp}.

Let $a_{1:n_1}(1:T) = \{a_{i}(t): 1 \leq i \leq n_1, 1 \leq t \leq T\}$, where $a_{i}(t)$ is the $i$-th row of the observed biadjacency matrix at time $t$. Similar to the edge LDP case, we consider (sequentially) interactive mechanisms where each row $Z_{i}(t)$ is allowed to depend on previous private information. i.e.
\[
 Z_{i}(t) | A_{i}(t) = a, \rightarrow Z_{ij}(t) \sim Q_{i}^{(t)}\big(\cdot | a , \rightarrow Z_{i}(t)\big),
\]
where the notation $\rightarrow Z_{i}(t) = \big(Z_{1:n_1}(< t), Z_{<i}(t)\big)$ contains all `previous' private information. Without loss of generality, we have fixed an order of interaction, i.e. at each time point $t$, the sequence of privatisation is $a_{i-1}(t) \rightarrow a_{i}(t)$, for $i = 2,\dotsc,n_1$.

\begin{defn}[Bipartite node $\alpha$-LDP]\label{def:node-ldp}
We say that the privacy mechanism $Q$ defined as
\[
Q(S|a_{1:n}(1:T)) = \int_{z_{1:n}(1:T) \in S} \prod_{t= 1}^T \prod_{i= 1}^n d Q_{i}^{(t)}(z_{i}{(t)} | a_{i}{(t)} , \rightarrow z_{i}(t))
\]
is bipartite node \(\alpha\)-LDP, if for any integer $1 \leq t \leq T$, any integer $1 \leq i \leq n_1$, any measurable set $S \subset \mathcal{Z}$ and any $a_{i}(t),a_i'(t) \in \{0, 1\}^{n_2}$, it holds that 
\begin{equation} \label{eq:nodeldp}
\frac{Q_i^{(t)}(Z_i{(t)} \in S| a_i(t), \rightarrow z_{i}(t) )}{Q_i^{(t)} (Z_i{(t)} \in S| a_i'(t), \rightarrow z_{i}(t))} \leq e^{\alpha}.
\end{equation}
\end{defn}

Different notions of node LDP have been studied in the literature. Our definition \eqref{eq:nodeldp} is consistent with 
\citep[e.g.][]{qin2017generating, wei2020asgldp} while some adopt the definition inherited from central DP allowing the neighbouring networks to have different dimension by either inclusion and deletion of one node \cite{kasiviswanathan2013analyzing,day2016publishing}. Several works consider the same constraint as \eqref{eq:nodeldp} under the name user-level LDP \citep[e.g.][]{levy2021learning, zhou2021locally} for different learning tasks. 

One appealing feature of bipartite graphs when considering node LDP is that the neighbouring data sets $x,x'$ can be protected independently for each node in $V_1$, whereas in a general graph, node LDP should account for the intrinsic symmetry of the adjacency matrix when defining neighbouring data sets \cite{imola2021locally}. Comparing the two LDP definitions we considered in this section, we see that in \Cref{def:edge-ldp} level $\alpha$ privacy is imposed to protect one edge, and in \Cref{def:node-ldp} level $\alpha$ privacy is imposed to protect $n_2$ edges. For the same privacy parameter $\alpha$, node privacy is a much more stringent constraint than edge privacy \citep[e.g.][]{qin2017generating,wei2020asgldp,imola2021locally}. 

\section{Fundamental limits in consistent change point localisation} \label{sec:lowerbound}

Recall that our task is to understand how the underlying distributions of dynamic networks change, especially to provide consistent change point estimators defined in \Cref{sec-problem-setup}, under certain form of LDP constraints.  Without the concern of privacy, dynamic IBN change point localisation is investigated in \cite{wang2021optimal}, where a scaling (namely the signal-to-noise ratio) is proposed to partition the whole parameter space into two regimes: a low signal-to-noise ratio regime (infeasibility regime) where no consistent estimator is guaranteed in a minimax sense, and a high signal-to-noise ratio regime where computationally-efficient algorithms are shown to produce consistent estimators.  Recall the model parameters $\kappa_0$ the minimal jump size, $\rho$ the entry-wise sparsity of networks, $n$ the network size and $\Delta$ the minimal spacing.  Without the presence of privacy constraints, the infeasibility regime \cite{wang2021optimal} is
\begin{equation}\label{eq-inf-reg-no}
    \kappa_0^2 \rho n \Delta \lesssim 1,
\end{equation}
which will serve as the benchmark for us to quantify the cost of privacy.

\textbf{The first model} we study is a dynamic IBN model (\Cref{def-ibn} and \Cref{asp:model-bi}), which is identical to the one studied in \cite{wang2021optimal}.  \Cref{lemma: edgelowerbound} demonstrates an infeasiblity regime of localising change points in such a model under the edge $\alpha$-LDP defined in \Cref{def:edge-ldp}.

\begin{lemma}[Edge $\alpha$-LDP]\label{lemma: edgelowerbound}
Let $\{A(t)\}_{t=1}^T \subset \{0, 1\}^{n \times n}$ be a sequence of adjacency matrices satisfying \Cref{asp:model-bi} with $K = 1$ and let $P^T_{\kappa_0,\Delta,n,\rho}$ denote their joint distribution.  Consider the class of distributions 
\[
    \mathcal{P} = \{P^T_{\kappa, \Delta, n, \rho}: \, \kappa_0^2  \leq \min \{ [68n\rho^2\Delta(e^\alpha-1)^2]^{-1},\,  1/4\},\, \Delta \leq T/3\}.
\]
Let $\mathcal{Q}_\alpha^{\mathrm{edge}}$ denote the set of all privacy mechanisms that satisfy the edge $\alpha$-LDP constraint in \Cref{def:edge-ldp}, for $\alpha \in (0, \min\{1,(2\rho)^{-1}\})$. We have that $$\inf_{Q \in \mathcal{Q}^{\mathrm{edge}}_\alpha} \inf_{\hat{\eta}} \sup_{P \in \mathcal{P}} \mathbb{E}_{P,Q}|\widehat{\eta}-\eta(P)| \geq \Delta/12,$$ where $\eta(P)$ denotes the change point location specified by distribution $P$, the first infimum is taken over all possible privacy mechanisms, the second infimum is taken over all measurable functions of the privatised data and the supremum is taken over all raw data's distributions in the class $\mathcal{P}$.
\end{lemma}

\Cref{lemma: edgelowerbound} studies an LDP minimax lower bound in the framework put forward by \cite{duchi2018minimax}.  It shows that for dynamic IBNs under edge $\alpha$-LDP, provided $\kappa_0^2 \rho^2 n \Delta (e^{\alpha} - 1)^2 \asymp \kappa_0^2 \rho^2 n \Delta \alpha^2 \lesssim 1$, the localisation error $\Delta^{-1}|\widehat{\eta} - \eta(P)| \geq 1/12$.  This leads to the infeasiblity regime 
\begin{equation}\label{eq-inf-reg-edge}
\kappa_0^2 \rho^2 n  \Delta \alpha^2 \lesssim 1.
\end{equation}

Comparing \eqref{eq-inf-reg-no} and \eqref{eq-inf-reg-edge}, any distribution in the regime \eqref{eq-inf-reg-no} also falls in the regime \eqref{eq-inf-reg-edge}, implying that imposing edge $\alpha$-LDP enlarges the infeasibility regime and makes the localisation task harder.  To be specific, the cost of preserving edge LDP comes from two fronts.

$\bullet$ The effective sample size is decreased from $\Delta$ to $\Delta \alpha^2$.  LDP's impact on the effective sample size is commonly observed in the literature over a wide range of problems \cite[e.g.][]{duchi2018minimax,butucea2020local,berrett2021changepoint,li2022robustness}.

$\bullet$ A more interesting and problem-specific cost of LDP is reflected by the role of the sparsity parameter~$\rho$, which power is raised to $\rho^2$ in \eqref{eq-inf-reg-edge} from $\rho$ in \eqref{eq-inf-reg-no}.  Despite that networks have been studied under LDP constraints, such result is the first time seen.  Similar effects have been observed in different problems under LDP constraint, including the impacts on dimensionality \citep[e.g.][]{berrett2021changepoint} and smoothness levels \citep[e.g.][]{li2022robustness}.  It is interesting to see that in a high-dimensional sparse network problem, this problem-specific cost of LDP appears on the sparsity parameter.

\textbf{The second model} we consider is a dynamic bipartite IBN model (\Cref{def-bibn} and \Cref{asp:model-bi}), the change point analysis of which is not seen in the literature, even without privacy concerns.  In addition to the rows and columns of bipartite IBNs denoting different entities, which is different from well-studied network models, we also allow potentially arbitrary within-row dependence. \Cref{lemma: nodelowerbound} establishes an infeasiblity regime of localising change points in such a model under the bipartite node $\alpha$-LDP defined in \Cref{def:node-ldp}.

\begin{lemma}[Bipartite node $\alpha$-LDP] \label{lemma: nodelowerbound}
Let $\{A(t)\}_{t=1}^T \subset \{0, 1\}^{n_1 \times n_2}$ be a sequence of biadjacency matrices satisfying \Cref{asp:model-bi} with $K = 1$ and let $P^T_{\kappa_0, \Delta, n_1, n_2, \rho}$ denote their joint distribution. Consider the class of distributions 
\[
    \mathcal{P} = \{P^T_{\kappa_0, \Delta, n_1, n_2, \rho}:\, \kappa_0^2 \leq \min \{[20 n_1^{1/2}\rho^2\Delta(e^\alpha-1)^2]^{-1}, \, 1/4\}, \, \Delta \leq T/3\}.
\]
Let $\mathcal{Q}_\alpha^{\mathrm{node}}$ denote the set of all privacy mechanisms that satisfy the bipartite node $\alpha$-LDP constraint in \Cref{def:node-ldp}, for $\alpha \in (0, \min\{1,(4\rho)^{-1}\})$. We have that $$\inf_{Q \in \mathcal{Q}^{\mathrm{node}}_\alpha} \inf_{\hat{\eta}} \sup_{P \in \mathcal{P}} \mathbb{E}_{P,Q}|\widehat{\eta}-\eta(P)| \geq \Delta/12,$$ where $\eta(P)$ denotes the change point location specified by distribution $P$.
\end{lemma}

In an LDP minimax framework, \Cref{lemma: nodelowerbound} shows that provided $\kappa_0^2 \rho^2 n_1^{1/2} \Delta \alpha^2 \lesssim 1$, the localisation error $\Delta^{-1}|\widehat{\eta} - \eta(P)| \geq 1/12$.  This leads to the infeasibility regime
\begin{equation}\label{eq-inf-reg-node}
\kappa_0^2 \rho^2 n_1^{1/2} \Delta \alpha^2 \lesssim 1.
\end{equation}

To compare \eqref{eq-inf-reg-edge} and \eqref{eq-inf-reg-node}, we first let $n_1 = n_2 = n$ in \Cref{lemma: nodelowerbound} for convenience.  The infeasibility regime under the node LDP reads as $\kappa_0^2 \rho^2 n^{1/2} \Delta \alpha^2 \lesssim 1$, which compared to \eqref{eq-inf-reg-edge} implies that the cost of node LDP is higher than the edge LDP. To further understand the difference between node LDP and edge LDP, we let $n = \sqrt{n_1n_2}$ in \Cref{lemma: edgelowerbound}.  The infeasibility regime under the edge LDP reads as $\kappa_0^2 \rho^2 (n_1n_2)^{1/2} \Delta \alpha^2 \lesssim 1$, which compared to \eqref{eq-inf-reg-node} highlights the difference of $n_2^{1/2}$, an extra cost of dimensionality. The extra cost captures the difference between privatising vectors with possibly correlated entries under node LDP and privatising discrete values under edge LDP.

\section{Consistent private network change point algorithms}\label{sec: consistentlocalisation}

We have established infeasibility regimes of change point localisation tasks under different network LDP constraints in \Cref{sec:lowerbound} and have understood how the privacy preservation makes the tasks fundamentally harder.  In this section, we provide polynomial-time private algorithms to obtain consistent change point estimators outside of the infeasibility regimes.  A private algorithm has two key ingredients: (1) a privacy mechanism and (2) an algorithm with privatised data as inputs.  For the two models we consider in this paper, we adopt the same change point localisation algorithm, while using different privacy mechanisms.  

The change point localisation algorithm we consider is the network binary segmentation (NBS) algorithm proposed and studied in \cite{wang2021optimal}.  It is shown that NBS provides consistent change point estimators without privacy concerns, under minimax optimal conditions.  For completeness, we include NBS in \Cref{algorithm:MWBS} and introduce the CUSUM statistic below.  For any form of data $\{X_i\}_{i = 1}^T$ and any integer triplet $0 \leq s < t < e \leq T$, the CUSUM statistic is defined as
\[
    \widetilde{X}^{(s,e)}(t) = \sqrt{\frac{e-t}{(e-s)(t-s)}}\sum_{i=s+1}^tX_i - \sqrt{\frac{t-s}{(e-s)(e-t)}}\sum_{i=t+1}^eX_i.
\]

\begin{algorithm}[!ht]
	\begin{algorithmic}
		\INPUT{$\{U(t)\}_{t= 1}^T, \{V(t)\}_{t= 1}^T \subset \mathbb{R}^{n_1 \times n_2}$, $\{ (\alpha_m, \beta_m)\}_{m=1}^M \subset [0, T]$, $\tau_1 > 0$}
		\For{$m = 1, \ldots, M$}  
			\State $[s_m', e_m'] \leftarrow [s,e]\cap [\alpha_m,\beta_m]$, 
			 $(s_m,e_m) \leftarrow [s_m' + 64^{-1} (e'_m - s'_m),e_m' - 64^{-1} (e_m' - s_m')] $  
			\If{$e_m - s_m \ge 1$}
				\State $b_{m} \leftarrow \arg\max_{t = s_m+1, \ldots, e_m-1}  ( \tilde U^{(s_m,e_m)}( t),\tilde V^{(s_m,e_m)} (t))$
				\State $a_m \leftarrow ( \widetilde U^{(s_m,e_m)} (b_m), \widetilde V^{(s_m,e_m)} (b_m))$
			\Else 
				\State $a_m \leftarrow -1$	
			\EndIf
		\EndFor
		\State $m^* \leftarrow \arg\max_{m = 1, \ldots, M} a_{m}$
		\If{$a_{m^*} > \tau$}
			\State add $b_{m^*}$ to the set of estimated change points
			\State NBS$((s, b_{m*}),\{ (\alpha_m,\beta_m)\}_{m=1}^M, \tau)$
			\State NBS$((b_{m*}+1,e),\{ (\alpha_m,\beta_m)\}_{m=1}^M,\tau)$
			
		\EndIf  
		\OUTPUT The set of estimated change points.
		\caption{Network Binary Segmentation. NBS$((s, e), \{ (\alpha_m,\beta_m)\}_{m=1}^M, \tau)$} \label{algorithm:MWBS}
	\end{algorithmic}
\end{algorithm} 

As pointed out in \cite{wang2021optimal}, two sequences of independent networks are required as inputs of \Cref{algorithm:MWBS} in order to estimate the Frobenius norm of an IBN.  In practice, one can split the data to even and odd indices to obtain two sequences of networks.

\subsection[]{Edge $\alpha$-LDP}  \label{sec:orgb73c267}

To privatise a dynamic IBN (\Cref{def-ibn}) under the edge $\alpha$-LDP, we apply the randomised response mechanism \citep{warner1965randomized} independently to every edge. The privacy guarantee follows by virtue of the the randomised response mechanism \cite{dwork2014algorithmic}. To be specific, given data $\{A(t)\}_{t = 1}^T \subset \{0, 1\}^{n \times n}$, let $\{U_{t, i, j}, 1 \leq i \leq j \leq n\}_{t = 1}^T$ be independent Unif$[0, 1]$ random variables that are independent of $\{A(t)\}_{t = 1}^T$.  For any $t \in \{1, \ldots, T\}$ and any integer pair $1 \leq i \leq j \leq n$, let the privatised data be $\{A'(t)\}_{t = 1}^T \subset \{0, 1\}^{n \times n}$ with
\begin{equation}\label{eq-sbc}
    A'_{ij}(t) = A'_{ji}(t) = \begin{cases}
        A_{ij}(t), & U_{t, i, j} \leq e^{\alpha}/(1 + e^{\alpha}), \\
        1 - A_{ij}(t), & \mbox{otherwise}.
    \end{cases}
\end{equation}
Note that due to the symmetry of the networks, each edge is only privatised once. Despite that we are dealing with a high-dimensional, sparse dynamic IBN model, with potentially multiple change points, \Cref{prop: edge upperbound} below shows that this, arguably simplest privacy mechanism not only provides consistent change point estimators, but also is optimal in terms of the signal-to-noise ratio condition required.  

\begin{thm}\label{prop: edge upperbound}
Let $\{A(t)\}_{t=1}^T$ and $\{B(t)\}_{t = 1}^T$ be two independent sequences of adjacency matrices satisfying \Cref{asp:model-bi}.  For an arbitrarily small $\xi > 0$ and an absolute constant $c_0 > 0$, assume that
\begin{equation} \label{eq: stn}
     \kappa_0^2 \rho^2 n \Delta \alpha^2 \geq c_0 \log^{2+\xi}(T).
\end{equation}

Let $\{\widehat{\eta}_k\}_{k = 1}^{\widehat{K}}$ be the output of the NBS algorithm, with inputs:

$\bullet$ $\{A'(t)\}_{t=1}^T$ and $\{B'(t)\}_{t = 1}^T$, privatised version of $\{A(t)\}_{t=1}^T$ and $\{B(t)\}_{t = 1}^T$ obtained through~\eqref{eq-sbc};  $\bullet$ $\{(\alpha_m,\beta_m)\}_{m=1}^M$, random intervals whose end points are drawn independently and uniformly from $\{1,\dotsc,T\}$ such that $\max_{m = 1}^M (\beta_m-\alpha_m) \leq C_R \Delta$, for some constant $C_R>3/2$; and $\bullet$ tuning parameter $\tau$ satisfying $c_1 n \log^{3/2}(T) < \tau < c_2 \kappa_0^2 n^2 \rho^2 \Delta \alpha^2$, where $c_1, c_2 > 0$ are absolute constants.

It holds with probability at least $1 - \exp\{\log(T/\Delta) - c_3M\Delta/T\} - c_4T^{-c_5}$ that
\[
    \widehat{K} = K \quad \mbox{and} \quad \max_{k = 1}^K|\widehat{\eta}_k - \eta_k| \leq c_6 \log(T) \{\sqrt{\Delta}/(\kappa_0 n \rho \alpha) + \sqrt{\log(T)}/(\kappa_0^2 \rho^2 n \alpha^2)\},
\]
where $c_3, c_4, c_5, c_6 > 0$ are absolute constants.
\end{thm}

\Cref{prop: edge upperbound} shows that, provided $M \gtrsim T \Delta^{-1} \log(T/\Delta)$, it holds with probability tending to one,
\begin{equation}\label{eq-consistent-1}
    \Delta^{-1} \max_{k = 1}^K|\widehat{\eta}_k - \eta_k| \lesssim \Delta^{-1} \log(T) \{\sqrt{\Delta}/(\kappa_0 n \rho \alpha) + \sqrt{\log(T)}/(\kappa_0^2 \rho^2 n \alpha^2)\} \to 0,
\end{equation}
where the second inequality is due to \eqref{eq: stn}.  Recalling the consistency definition in \Cref{sec-problem-setup}, \eqref{eq-consistent-1} implies the consistency of NBS with randomised response privacy mechanism under edge $\alpha$-LDP.  

In view of the condition \eqref{eq: stn} and the edge LDP infeasibility regime \eqref{eq-inf-reg-edge}, up to a logarithmic factor, we unveil a phase transition with boundary $\kappa_0^2 \rho^2 n \Delta \alpha^2 \asymp 1$ and show that the randomised response mechanism is optimal in the minimax sense. This is conceptually interesting since, as pointed out in \cite{qin2017generating}, the privatised network obtained by \eqref{eq-sbc} leads to a dense graph even though the original graph may be sparse and therefore does not represent the original graph well. However, our result shows that this simple non-interactive mechanism is the best one can do for change point localisation, even among interactive mechanisms.

\subsection[]{Bipartite node $\alpha$-LDP}\label{sec: nodeLDPupper}

To privatise a dynamic bipartite IBN (\Cref{def-bibn}) under the bipartite node $\alpha$-LDP, we apply the privacy mechanism developed in \citet{duchi2013local,duchi2018minimax} for privatising vectors with bounded $\ell_\infty$ norm to each row of the biadjacency matrices. This privacy mechanism has been used in the analysis of mean estimation \cite[e.g.][]{duchi2018minimax}, nonparametric density estimation \cite[e.g.][]{duchi2018minimax,li2022robustness} and exact support recovery \cite[e.g.][]{butucea2020sharp} problems under LDP.

Given data $\{A(t)\}_{t = 1}^T \subset \{0, 1\}^{n_1 \times n_2}$, let $\{U_{t, i}\}_{t = 1, i = 1}^{T, n_1}$ be independent Unif$[0, 1]$ random variables that are independent of $\{A(t)\}_{t = 1}^T$ and let $\{\tilde{A}_{ij}(t)\}_{t = 1, i = 1, j = 1}^{T, n_1, n_2}$ be random variables satisfying 
\[
    \mathbb{P}\{\tilde{A}_{ij}(t) = 1 | A_{ij}(t)\} = 1 - \mathbb{P}\{\tilde{A}_{ij}(t) = -1| A_{ij}(t)\} = \{1 + A_{ij}(t)\}/2.  
\]
Let
\[
    B = C_{n_2} (e^\alpha+1)/(e^\alpha-1) \quad \mbox{with} \quad C_{n_2}^{-1} =\begin{cases}  \frac{1}{2^{n_2-1}} \binom{n_2-1}{(n_2-1)/2}, & n_2 \,\mathrm{mod}\, 2 \equiv 1, \\
        \frac{1}{2^{n_2-1}+ \frac{1}{2}\binom{n_2}{n_2/2}} \binom{n_2-1}{n_2/2}, & n_2 \,\mathrm{mod}\, 2 \equiv 0. 
        \end{cases}
\]
The privatised data $\{A'(t)\}_{t = 1}^T \subset \{0, 1\}^{n_1 \times n_2}$ are obtained by sampling 
\begin{equation}\label{eq:duchisampling}
    A'_i(t) \sim \begin{cases}
        \mathrm{Unif}\Big(z \in \{B, -B\}^{n_2} | \sum_{j = 1}^{n_2} z_i \tilde{A}_{ij}(t) \geq 0\Big), & U_{t, i} \leq e^{\alpha}/(1+e^{\alpha}),\\
        \mathrm{Unif}\Big(z \in \{B, -B\}^{n_2} | \sum_{j = 1}^{n_2} z_i \tilde{A}_{ij}(t) \leq 0\Big), & \mbox{otherwise}.
    \end{cases}
\end{equation}

Note that $\|A_i(t)\|_\infty = 1$ for any $i = 1,\dotsc,n_1$ and $t = 1,\dotsc,T.$ Applying (26) in \cite{duchi2018minimax} with $d = n_2$ guarantees that $A'_i(t)$ is an $\alpha$-private version of $A_i(t)$ and therefore satisfies the bipartite node $\alpha$-LDP constraint. In \Cref{thm: nodeupperbound}, we demonstrate that NBS with inputs obtained through \eqref{eq:duchisampling} is consistent in localising change points under bipartite node $\alpha$-LDP constraint.

\begin{thm}\label{thm: nodeupperbound}
Let $\{A(t)\}_{t=1}^T$ and $\{B(t)\}_{t = 1}^T$ be two independent sequences of biadjacency matrices satisfying \Cref{asp:model-bi}.  For an arbitrarily small $\xi > 0$ and an absolute constant $c_0 > 0$, assume that
\begin{equation}
\label{eq: node_stn}
     \kappa^2_0 \rho^2 \min \{\sqrt{n_1/n_2}, \, n_1/n_2\} \Delta \alpha^2 \geq c_0 \log^{2+\xi}(Tn_1n_2).
\end{equation}

Let $\{\widehat{\eta}_k\}_{k = 1}^{\widehat{K}}$ be the output of the NBS algorithm with inputs:

$\bullet$ $\{A'(t)\}_{t=1}^T$ and $\{B'(t)\}_{t = 1}^T$, privatised version of $\{A(t)\}_{t=1}^T$ and $\{B(t)\}_{t = 1}^T$ obtained through~\eqref{eq:duchisampling};  $\bullet$ $\{(\alpha_m,\beta_m)\}_{m=1}^M$, random intervals whose end points are drawn independently and uniformly from $\{1,\dotsc,T\}$ such that $\max_{m = 1}^M (\beta_m-\alpha_m) \leq C_R \Delta$, for some constant $C_R>3/2$; and $\bullet$ tuning parameter $\tau$ satisfying $c_1 n_2\alpha^{-2} \log^2(Tn_1n_2)\max\{\sqrt{n_1n_2},n_2\} < \tau < c_2 \kappa_0^2n_1n_2\rho^2 \Delta$, where $c_1, c_2 > 0$ are absolute constants.

It holds with probability at least $1 - \exp\{\log(T/\Delta) - c_3M\Delta/T\} - c_4 T^{-c_5}$ that $\widehat{K} = K$ and
\[%begin{equation}\label{eq: localisation_node}
     \max_{k = 1}^K|\widehat{\eta}_k - \eta_k| \leq c_6 \log(Tn_1n_2)\Bigg(\frac{\sqrt{\Delta }}{\kappa_0\rho\alpha}\sqrt{\frac{n_2}{n_1}}+\frac{{\log(Tn_1n_2)}}{\rho^2\alpha^2\kappa_0^2}\max\Bigg\{\sqrt{\frac{n_2}{n_1}},\frac{n_2}{n_1}\Bigg\}\Bigg),
\]%end{equation}
where $c_3, c_4, c_5, c_6 > 0$ are absolute constants.
\end{thm}

\Cref{thm: nodeupperbound} shows that, provided $M \gtrsim T \Delta^{-1} \log(T/\Delta)$, NBS with privatised inputs through channel \eqref{eq:duchisampling} is consistent.  When $n_2 \asymp 1$, the signal-to-noise ratio condition \eqref{eq: node_stn} and the infeasibility regime~\eqref{eq-inf-reg-node} demonstrate a phase transition with boundary $\kappa_0^2 \rho^2 n_1^{1/2} \Delta \alpha \asymp 1$, up to a logarithmic factor.  When $n_2$ is allowed to diverge, a gap between the infeasibility regime \eqref{eq-inf-reg-node} and \eqref{eq: node_stn} - the regime where our proposed method is deemed to be consistent - emerges.  The larger $n_2$ is, the larger the gap is.  It is interesting to understand further what happens within the gap and we leave this as an open problem, which echos the challenging problems in high-dimensional statistical inference under LDP.  

To conclude this section, we would like to present some result of independent interest.  It is studied in the existing literature \cite[Appendix I.3 in][]{duchi2018minimax} that the privatised output from \eqref{eq:duchisampling} is unbiased, i.e.~$\mathbb{E}\{A_i'(t)\} = \mathbb{E}\{A_i(t)\}$, while the covariance structure of the privatised output is unknown.  In \Cref{lemma: covariance}, we carefully analyse the covariance matrix of the privatised output and provide an upper bound on its operator norm. Due to its independent interest, we denote the raw data vector as $V = (V_i) \in \mathbb{R}^d$ and denote its privatised output obtained through \eqref{eq:duchisampling} as $Z = (Z_i) \in \mathbb{R}^d$.

\begin{lemma} \label{lemma: covariance}
For any random vector $V \in \mathbb{R}^d$ with $\|V\|_\infty \leq 1$, we have that
\begin{equation}\label{eq:cov1}
    \mathrm{Var}(Z_i) = B^2 - \{\E(V_i)\}^2, \quad i = 1, \ldots, d;
\end{equation}
and
\begin{equation}\label{eq:cov2}
    \mathrm{Cov}(Z_i, Z_j) = \begin{cases}
    -\E(V_i)\E(V_j), & d \mbox{ mod } 2 \equiv 1,\\
    -\E(V_i)\E(V_j) - \frac{C_{d,\alpha}}{d^{1/2}\alpha^2}\E(V_iV_j), & d \mbox{ mod } 2 \equiv 0,
\end{cases} \quad \forall \, i \neq j,
\end{equation}
where where $C_{d,\alpha} \in [C_0,C_1]$  for some absolute constants $C_1>C_0>0$.  Letting $\Sigma_Z$ be the covariance matrix of $Z$, it holds that
\begin{equation}\label{eq: operatornorm}
    \|\Sigma_Z\| \leq \begin{cases}
    B^2+\|\E (V)\|_2^2, & d \mbox{ mod } 2 \equiv 1 \\
    B^2+\|\E (V)\|_2^2+\frac{c\sqrt{d}}{\alpha^2}\sqrt{\max_{i,j}\E(V_iV_j)} & d \mbox{ mod } 2 \equiv 0, 
    \end{cases}
\end{equation}
where $c > 0$ is an absolute constant.
\end{lemma}

\section{Numeric results}

We generate a sequence of $T$ independent IBNs (Defintion \ref{def-ibn}) or bipartite IBNs with independent edges (Definition \ref{def-bibn}) when considering node LDP, with the network size $n_1 = n_2 = n = 50$ and entrywise sparsity level $\rho = 0.4$.  There is one and only one change point with a balanced spacing, i.e.~the change point $\eta = \Delta = T/2$, where $\Delta$ is the minimal spacing.  The expectations of the adjacency matrix before and after change point are $\Theta_{\mathrm{pre}} = 0.1 \mathbf{1}_{n \times n}$ and $\Theta_{\mathrm{post}} = 0.4  \mathbf{1}_{n \times n}$, respectively, where $\mathbf{1}_{n \times n} \in \mathbb{R}^{n \times n}$ has all entries being one.  The normalised jump size is therefore $\kappa_0 = \|\Theta_{\mathrm{post}} - \Theta_{\mathrm{pre}}\|_{\mathrm{F}}/(n\rho) = 0.75$.  We consider different the minimal spacing $\Delta$ and privacy budget $\alpha$ in the simulations.

We use a simplified version of NBS algorithm (\Cref{algorithm:MWBS}) based on the binary segmentation procedure \cite[e.g.][]{vostrikova1981detecting}. For small number of change points, our theory still holds for this computationally less demanding algorithm. The thresholding tuning parameter, above which change points are declared, is fixed to be $n\log^{1.5}(T)/10$, $n\log^{1.5}(T)/30$ and $n^2\log^2(n^2T)/10$ in the no privacy, edge LDP and node LDP cases, respectively. 

Let the estimated set of change points be $\{\hat{\eta}_i\}_{i = 1}^{\hat{K}}$ and the true change points be $\eta$.  We use $\max_{i}|\hat{\eta}_i - \eta|/\Delta \in [0,1]$ to evaluate the performances.  If no change point is returned, we output one. This is the same as using the scaled two sided Hausdorff distance $d_H(S_1,S_2)/\Delta$ as the metric \cite[e.g.][]{li2021adversarially,wang2020univariate} and we expect it to diminish as $\Delta$ grows. For any subset $S_1$, $S_2 \subset \mathbb{Z}$, the Hausdorff distance $d_H(S_1,S_2)$ between $S_1$ and $S_2$ is defined as
    \[
     d_H(S_1,S_2) = \max\left\{\max_{s_1 \in S_1}\min_{s_2\in S_2}|s_1-s_2|, \max_{s_2\in S_2}\min_{s_1\in S_1} |s_1-s_2|\right\}.
   \]
The sets $S_1$ and $S_2$ correspond to the set of true change points and estimated change points. If one of $S_1$ and $S_2$ is $\emptyset$, then we use the convention $d_H(S_1,S_2) = \Delta$.

The result is collected in \Cref{fig:k0-0.75}.  Without any privacy constraint, i.e.~using raw data, the change can be easily detected with $\Delta$ as small as $7$.  Imposing privacy guarantee requires a larger $\Delta$ to consistently localise the change points. The theoretical cost is quantified by our theory under both edge LDP and node LDP. We can see from the three plots in the first row that for the same sample size, the performance deteriorates as $\alpha$ decreases under edge LDP.  The node LDP is a more stringent requirement, compared to the edge LDP. From the three plots on the second row, we can see that, with the same sample size, the change can be perfectly localised with no error in the no privacy case, and very well localised under edge LDP with $\alpha = 0.1$, but in order to obtain a reasonable estimator, the node information can only be protected at level $\alpha = 1$.

\begin{figure}[ht!!]
    \centering
    \includegraphics[width = 0.9\linewidth]{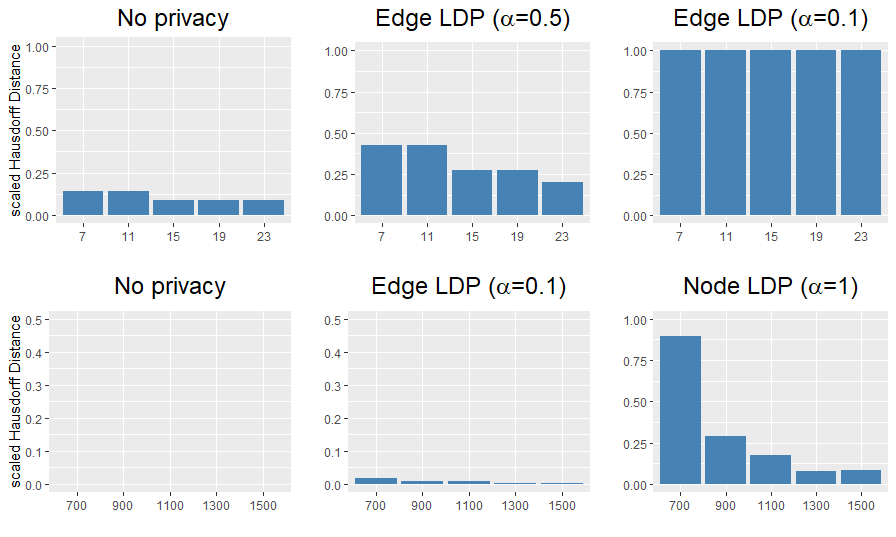}
    \caption{Simulation results. The median of the scaled Hausdorff distance $d_H(S_1,S_2)/\Delta$ over $100$ repetitions are plotted against varying minimal spacing $\Delta$ on the $x$-axis, under different privacy constraints. This setting has $\kappa_0 = 0.75$, $\rho = 0.4$, $n_1 = n_2 = n = 50$}
    \label{fig:k0-0.75}
\end{figure}

\section{Conclusion}\label{conclusion}

In this paper, we studied network change point localisation problems under two forms of LDP constraints. New signal-to-noise conditions \eqref{eq: stn} and \eqref{eq: node_stn} are derived and by comparing with the non private counterpart, we quantify the cost of privacy as discussed in Sections~\ref{sec:lowerbound} and \ref{sec: consistentlocalisation}. A change in the scaling of sparsity parameter in the private signal to noise conditions reveals a new challenge of learning dynamic networks with possibly sparse and correlated entries. The results are summarised in the table below, where for clarity we ignored logarithmic factors and consider $n_1 = n_2 = n$ in the bipartite node LDP case. 

\begin{table}[ht]
\centering
\begin{tabular}{cccc}
\toprule
 No privacy \cite{wang2021optimal}         & Edge LDP [\eqref{eq-inf-reg-edge}\&\eqref{eq: stn}] &  Node LDP  & Node LDP  \\ 
 & & lower bound \eqref{eq-inf-reg-node} & upper bound \eqref{eq: node_stn} \\
 
 \midrule
$\kappa_0^2 \rho \asymp \frac{1}{n\Delta}$ & $\kappa_0^2\rho^2 \asymp \frac{1}{n\Delta \alpha^2} $ & $\kappa_0^2 \rho^2 \lesssim \frac{1}{\sqrt{n}\Delta\alpha^2} $ &$\kappa_0^2 \rho^2 \gtrsim \frac{1}{\Delta\alpha^2}  $                \\ 
\bottomrule
\label{table: boundary}
\end{tabular}
\end{table}

The question left open is closing the gap in the node LDP case. From the lower bound perspective, we identify a technical challenge in controlling the $\chi^2$-divergence between mixtures of private distributions. Although some techniques have been developed for discrete distributions \cite[e.g.][]{berrett2020locally,acharya2020inference}, the counterpart for high-dimensional discrete distributions is still largely unexplored. As for the upper bound, our current method is non-interactive. Since different entries in our network model follow different distributions, we therefore expect that allowing interaction within networks cannot improve the signal to noise ratio condition, while interaction across time points requires novel methodology that can handle temporal dependence, account for the decay of privacy and is suitable for the task of change point localisation. We leave that as our future work.

\section*{Acknowledgements and Disclosure of Funding}
The authors would like to thank Harry Giles for helpful discussion and suggesting the idea behind the proof of \Cref{prop: edge upperbound}. TBB acknowledges the support of an Engineering and Physical Sciences Reseach Council (EPSRC) New Investigator Award EP/W016117/1.  YY acknowledges the support of an EPSRC Standard Grant EP/V013432/1.

\newpage
\bibliographystyle{agsm}
\bibliography{ref}

\newpage
\section*{Supplementary material}

The supplementary material contains proofs of the main results in \Cref{sec:lowerbound} and \Cref{sec: consistentlocalisation}. 

\appendix

\section[]{Proofs of results in \Cref{sec:lowerbound}}

\begin{proof}[Proof of \Cref{lemma: edgelowerbound}]
Let $\kappa^2 = \frac{n}{68(e^\alpha-1)^2\Delta}$, $ v \in \{1,-1\}^n$ and $P_v^T$ be the joint distribution of a collection of independent adjacency matrices $\{A(t)\}_{t=1}^T$ such that 
\[
\mathbb{E}[A_{ij}(t)] = \rho/2+ \frac{\kappa}{n}(vv^\top)_{ij}\quad 1\leq i\leq j \leq n, \qquad t \in \{1,\dotsc,\Delta\} 
\]
and
\[
\mathbb{E}[A_{ij}(t)] = \rho/2 \quad 1\leq i\leq j \leq n, \qquad t \in \{\Delta+1,\dotsc,T\}.
\]
The distribution of each network at time $t$ is denoted by $P_{v,t} = \prod_{1\leq i \leq j\leq n}P_{v,t,(i,j)}$. Note that $\eta(P_v^T) = \Delta$, $\|A(\Delta)-A(\Delta+1)\|_F^2 = \kappa^2$, and $\kappa_0^2 = \frac{1}{68n\rho^2\Delta(e^\alpha-1)^2}$. We are constrained by $\kappa/n < \rho/2$, which is equivalent to 
$\kappa_0^2 \leq 1/4$. Therefore, for each $v$, we have $P_v^T \in \mathcal{P}$. Similarly, let $\tilde{P}_v^T$ be the joint distribution of a collection of independent adjacency matrices $\{A(t)\}_{t=1}^T$ such that 
\[
\mathbb{E}[A_{ij}(t)] = \rho/2+ \frac{\kappa}{n}(vv^\top)_{ij}\quad 1\leq i\leq j \leq n, \qquad t \in \{T-\Delta+1,\dotsc,T\} 
\]
and
\[
\mathbb{E}[A_{ij}(t)] = \rho/2 \quad 1\leq i\leq j \leq n, \qquad t \in \{1,\dotsc,T-\Delta\}.
\]
The distribution of each network at time $t$ is denoted by $\tilde{P}_{v,t} = \prod_{1\leq i \leq j\leq n}\tilde{P}_{v,t,(i,j)}$. Note that $\eta(\tilde{P}_v^T) = T-\Delta$ and $\tilde{P}_v^T \in \mathcal{P}$ for each $v$. Also, $|\eta({P}_v^T) - \eta(\tilde{P}_v^T)| = T-2\Delta$ for each $v$. 
Further, let $Z^T_v$ and $\tilde{Z}^T_v$ be the corresponding joint private distribution generated via some \emph{edge} $\alpha$ LDP mechanism $Q$, i.e.
\begin{align}\label{eq: notationlowerbound}
    Z^T_v(\cdot) = \int Q(\cdot|a_{1:n,1:n}(1:T)) dP^T_v(a_{1:n,1:n}(1:T)) 
\end{align}
and $\tilde{Z}_v^T$ takes a similar form. Write $Z^T = \frac{1}{2^n}\sum_{v \in \{1,-1\}^n}Z_v^T $ and $\tilde{Z}^T = \frac{1}{2^n}\sum_{v \in \{1,-1\}^n}\tilde{Z}_v^T$. 
It follows from Le Cam's Lemma \cite[e.g.][]{yu1997assouad} that for $\Delta \leq T/3$
\[
\mathcal{R}_{n,\alpha}(\eta(\mathcal{P})) \geq \inf_{Q \in \mathcal{Q}_\alpha^{\text{edge}}}\frac{T}{6}(1-\mathrm{TV}(\tilde{Z}^T,Z^T)) \geq \inf_{Q \in \mathcal{Q}_\alpha^{\text{edge}}} \frac{\Delta}{6} (1-\mathrm{TV}(\tilde{Z}^T,Z^T)).
\]
To simplify the problem, we write $P_0^T$ as the joint distribution of independent and identically distributed adjacency matrices $\{B(t)\}_{t=1}^{T}$ such that $\mathbb{E}[(B(t))_{ij}] = \rho/2$ for $ 1\leq i \leq j \leq n$. The corresponding marginal distribution of the privatised data is denoted by ${Z}_0^T$. Now, notice that by triangle inequality and the symmetry of our construction, we have 
\[
\mathrm{TV}(\tilde{Z}^T,Z^T) \leq 2 \mathrm{TV}(Z_0^T,Z^T) \leq \sqrt{2\chi^2(Z^T,Z_0^T)},
\]
where the last inequality is due to \cite[][eq.(2.27)]{tsybakov2009introduction}. 
In the rest of the proof, we will show that with our choice $\kappa^2\Delta(e^\alpha-1)^2 = n/68$, we have $\chi^2(Z^T,Z_0^T) \leq 1/8$ and therefore $\mathcal{R}_{n,\alpha}(\eta(\mathcal{P})) \geq \Delta/12$ as claimed. 

We need some notations. We write $m_v(z_{1:n,1:n}(1:T))$ as the density of the measure $Z^T_v$ and similarly for the notation $m_0(z_{1:n,1:n}(1:T))$. We also write $\Gamma = (\kappa/n) uu^\top$, where $u \in \{-1,1\}^n$, $\Lambda = (\kappa/n) vv^\top$ and use $U, V \in \mathbb{R}^n$ to denote two independent random vectors with entries being independent Rademacher random variables. With these notations, we compute 
\begin{align*}
    \chi^2(Z^T,Z_0^T) + 1 &= \frac{1}{4^n}\sum_{u,v \in \{-1,1\}^n}\E_{Z^T_0}\Bigg(\frac{m_v(Z_{1:n,1:n}(1:T)) m_u(Z_{1:n,1:n}(1:T))}{m_0(Z_{1:n,1:n}(1:T)) m_0(Z_{1:n,1:n}(1:T))}\Bigg) \\
    &\leq \frac{1}{4^n}\sum_{u,v \in \{-1,1\}^n} \Bigg[\prod_{t = 1}^{\Delta}\prod_{1\leq i\leq j\leq n}\exp(2\Gamma_{ij}\Lambda_{ij}(e^\alpha-1)^2) \Bigg]\\
    &= \E_{U,V}\Bigg[\exp\Big(\frac{2\Delta\kappa^2(e^\alpha-1)^2}{n^2}(U^\top V)^2\Big)\Bigg] \\
    & = \E_V\Bigg[\exp\Big(\frac{2\Delta\kappa^2(e^\alpha-1)^2}{n^2}(\mathbf{1}^\top V)^2\Big)\Bigg]
\end{align*}
where the inequality is due to \Cref{lemma:controlchisquare} and $\mathbf{1} \in \mathbb{R}^n$ is a vector of $1$'s.

Let $\epsilon_n = (\mathbf{1}^{\top}V)^2/n^2$, then 
\begin{align*}
    \E_V\Bigg[\exp\Big(2\Delta\kappa^2(e^\alpha-1)^2 \epsilon_n\Big)\Bigg] &= \int_{0}^\infty\mP\Big(\exp(2\Delta\kappa^2(e^\alpha-1)^2\epsilon_n) \geq u\Big) du \\
    & \leq 1+ \int_1^\infty\mP\Bigg(\epsilon_n \geq \frac{\log(u)}{2\kappa^2\Delta(e^\alpha-1)^2}\Bigg) du \\
    &\leq 1+\int_1^\infty 2\exp\Big(- \log(u)\frac{n}{4\kappa^2\Delta(e^\alpha-1)^2}\Big)du \\
    & \leq 1+\frac{2}{\frac{n}{4\kappa^2\Delta(e^\alpha-1)^2}-1},
\end{align*}
where the second inequality is Hoeffding's inequality \cite[][Theorem 2.2.6]{vershynin2018high} and the last inequality holds if 
\[
\frac{n}{4\kappa^2\Delta(e^\alpha-1)^2} > 1
\]
For $\chi^2(Z^T,Z_0^T) \leq 1/8$, it is sufficient to take 
\[
\frac{n}{\kappa^2\Delta(e^\alpha-1)^2} \geq 68, 
\]
which completes the proof. 
\end{proof}

\begin{lemma}\label{lemma:controlchisquare}
With $\Gamma = (\kappa/n) uu^\top$ and $\Lambda = (\kappa/n) vv^\top$, it holds that
\[
\E_{Z^T_0}\Bigg(\frac{m_v(Z_{1:n,1:n}(1:T)) m_u(Z_{1:n,1:n}(1:T))}{m_0(Z_{1:n,1:n}(1:T)) m_0(Z_{1:n,1:n}(1:T))}\Bigg) \leq \prod_{t = 1}^{\Delta}\prod_{1\leq i\leq j\leq n}\exp(2\Gamma_{ij}\Lambda_{ij}(e^\alpha-1)^2). 
\]
   
\end{lemma}
\begin{proof} 
Recall that the notation $\rightarrow z_{ij}(t)$ denote the collection of private information that can be used to generate $z_{ij}(t)$ according to \eqref{eq:orderinteraction}. Note that 
\[
\frac{m_v(z_{1:n,1:n}(1:T)) m_u(z_{1:n,1:n}(1:T))}{m_0(z_{1:n,1:n}(1:T)) m_0(z_{1:n,1:n}(1:T))} =  \prod_{j = 1}^n \prod_{i = 1}^j \prod_{t = 1}^T  \frac{m_v(z_{ij}(t)|\rightarrow z_{ij}(t)) m_u(z_{ij}(t)|\rightarrow z_{ij}(t))}{m_0(z_{ij}(t)|\rightarrow z_{ij}(t)) m_0(z_{ij}(t)|\rightarrow z_{ij}(t))}. 
\]
For $\Delta < t \leq T$, by construction we have for any $1 \leq i \leq j \leq n$
\begin{align*}
    m_v\bigl(z_{ij}(t)|\rightarrow z_{ij}(t)\bigr) &= m_u(z_{ij}(t)|\rightarrow z_{ij}(t)) = m_0(z_{ij}(t)|\rightarrow z_{ij}(t)) \\
    &= q(z_{ij}(t)|a_{ij}(t) = 1, \rightarrow z_{ij}(t))\rho/2 + (1-\rho/2) q(z_{ij}(t)|a_{ij}(t) = 0, \rightarrow z_{ij}(t)).
\end{align*}
Thus, 
\begin{align}\label{eq:interactivechisquare}
    &\E_{Z^T_0}\Bigg(\frac{m_v(Z_{1:n,1:n}(1:T)) m_u(Z_{1:n,1:n}(1:T))}{m_0(Z_{1:n,1:n}(1:T)) m_0(Z_{1:n,1:n}(1:T))}\Bigg) = \E_{Z^{\Delta}_0}\Bigg(\frac{m_v(Z_{1:n,1:n}(1:\Delta)) m_u(Z_{1:n,1:n}(1:\Delta))}{m_0(Z_{1:n,1:n}(1:\Delta)) m_0(Z_{1:n,1:n}(1:\Delta))}\Bigg) \nonumber \\
    & = \textstyle \E_{Z^\Delta_0}\Bigg( \frac{m_v(\rightarrow Z_{nn}(\Delta)) m_u(\rightarrow Z_{nn}(\Delta))}{m_0(\rightarrow Z_{nn}(\Delta)) m_0(\rightarrow Z_{nn}(\Delta))} \E_{Z^\Delta_0} \left(\frac{m_v(Z_{nn}(\Delta)|\rightarrow Z_{nn}(\Delta)) m_u(Z_{nn}(\Delta)|\rightarrow Z_{nn}(\Delta))}{m_0^2(Z_{nn}(\Delta)|\rightarrow Z_{nn}(\Delta))}\bigg|\rightarrow Z_{nn}(\Delta)\right)\Bigg)
\end{align}

Note that for any $1 \leq i \leq j \leq n$ and $ 1\leq t \leq \Delta$, writing $z$ as the privatised data $z_{ij}(t)$, $q_1(z) = q_{ij}^{(t)}(z|a_{ij}(t)  = 1, \rightarrow z)$ and $q_0(z) = q_{ij}^{(t)}(z|a_{ij}(t)  = 0, \rightarrow z)$, we have 
\begin{align*}
    &\E_{Z^t_0} \left(\frac{m_v(Z_{ij}(t)|\rightarrow Z_{ij}(t)) m_u(Z_{ij}(t)|\rightarrow Z_{ij}(t))}{m_0^2(Z_{ij}(t)|\rightarrow Z_{ij}(t))}\bigg|\rightarrow Z_{ij}(t)\right) \\
    &= \int \frac{\Big[q_1(z)(\rho/2+\Gamma_{ij})+q_0(z)(1-\rho/2-\Gamma_{ij})\Big]\Big[q_1(z)(\rho/2+\Lambda_{ij})+q_0(z)(1-\rho/2-\Lambda_{ij})\Big]}{\rho  q_1(z)+q_0(z)(1-\rho)} dz \nonumber \\
    & =\bigints \frac{\splitfrac{\Big[(q_1(z)-q_0(z))\rho/2+q_0(z)+\Gamma_{ij}(q_1(z) - q_0(z))\Big]\Big[(q_1(z)-q_0(z))\rho/2+q_0(z)+}{\Lambda_{ij}(q_1(z) - q_0(z)\Big]}}{(q_1(z)-q_0(z))\rho/2+q_0(z)} dz \nonumber \\
     &= (I)+(II)+(III)
\end{align*}
where 
\begin{align*}
    (I) &= \int [(q_1(z)-q_0(z))\rho/2+q_0(z)] dz = 1 \\
    (II) & = \int (q_1(z)-q_0(z))(\Gamma_{ij}+\Lambda_{ij}) dz = 0 
\end{align*}
since $q_1(z)$ and $q_0(z)$ are densities of regular conditional probability distributions. Also,
\begin{align*}
    (III) & = \Gamma_{ij}\Lambda_{ij}\int \frac{(q_1(z) - q_0(z))^2}{(q_1(z)-q_0(z))\rho+q_0(z)}dz = \Gamma_{ij}\Lambda_{ij}\int \frac{(q_0(z))^2(q_1(z)/q_0(z)-1)^2}{(q_1(z)-q_0(z))\rho+q_0(z)}dz \\
    & = \Gamma_{ij}\Lambda_{ij}C_\alpha 
\end{align*}
with $0 \leq C_\alpha \leq 2(e^\alpha-1)^2$ where the last equality is due to \Cref{lemma: loweredge} and $q_1(x)/q_0(x) \in [e^{-\alpha},e^\alpha]$. Using the inequality $1+2\Gamma_{ij}\Lambda_{ij}(e^\alpha-1)^2 \leq \exp(2\Gamma_{ij}\Lambda_{ij}(e^\alpha-1)^2)$, we obtain 
\[
\E_{Z^\Delta_0} \left(\frac{m_v(Z_{nn}(\Delta)|\rightarrow Z_{nn}(\Delta)) m_u(Z_{nn}(\Delta)|\rightarrow Z_{nn}(\Delta))}{m_0^2(Z_{nn}(\Delta)|\rightarrow Z_{nn}(\Delta))}\bigg|\rightarrow Z_{nn}(\Delta)\right) \leq \exp(2\Gamma_{nn}\Lambda_{nn}(e^\alpha-1)^2).
\]
 Continue factoring \eqref{eq:interactivechisquare} according to the order of interaction defined in \eqref{eq:orderinteraction} and repeatedly applying the above argument for each entry $Z_{ij}(t)$ yields the claimed result.

\end{proof}

\begin{lemma}\label{lemma: loweredge}
When $\alpha \leq 1$ and $\alpha \rho \leq 1/2$, then
\[
(q_1(z)-q_0(z))\rho/2+q_0(z) \geq q_0(z)/2
\]
\end{lemma}
\begin{proof}[Proof of \Cref{lemma: loweredge}]
Using the facts $q_1(z)/q_0(z) \geq e^{-\alpha}$, and $e^\alpha-1 \geq 1-e^{-\alpha}$, we obtain
\begin{align*}
    (q_1(z)-q_0(z))\rho/2+q_0(z) &\geq q_0(z)\Big(\frac{q_1(z)}{q_0(z)}-1\Big)\rho/2+q_0(z) \geq q_0(z)(1-\rho(1-e^{-\alpha})/2) \\
    & \geq q_0(z)(1-\rho(e^\alpha-1)/2) \geq q_0(z)(1-\alpha \rho) \geq q_0(z)/2,
\end{align*}
where in the last two inequalities we use $\alpha \leq 1$ and $\alpha \rho \leq 1/2$ respectively. 
\end{proof}

\begin{proof}[Proof of \Cref{lemma: nodelowerbound}]
The proof parallels the structure of the proof of \Cref{lemma: edgelowerbound}, so we are somewhat more terse. Let $\kappa^2 = \frac{\sqrt{n_1}n_2}{20(e^\alpha-1)^2\Delta}, v \in \{1,-1\}^{n_1}$, $P_v^T$ be the joint distribution of a collection of independent adjacency matrices $\{A(t)\}_{t=1}^T$ such that for $ t \in \{1,\dotsc,\Delta\}$, 
\[
\mathbb{E}[A_{i1}(t)] = \rho/2+ \frac{\kappa}{\sqrt{n_1n_2}}v_{i}\quad 1\leq i \leq n_1, \quad \text{and} \quad A_{i1}(t) = \dotsc = A_{in_2}(t),
\]
and for $t \in \{\Delta+1,\dotsc,T\}$,
\[
\mathbb{E}[A_{i1}(t)] = \rho/2 \quad \text{and} \quad A_{i1}(t) = \dotsc = A_{in_2}(t).
\]
In words, within each network, the entries of each row are \emph{identical}. In particular, we have for any $ 1\leq i\leq n_1, 1\leq j \leq n_2$,
\begin{align*}
    \mathbb{E}[A_{ij}(t)] = \rho/2+ \frac{\kappa}{\sqrt{n_1n_2}}v_{i}  \quad 1\leq t \leq \Delta, \qquad
     \mathbb{E}[A_{ij}(t)] = \rho/2 \quad \Delta+1\leq t \leq T
\end{align*}
and 
\begin{align*}
    \mP(A_{i}(t) = \mathbf{1}) &= 1-\mP(A_{i}(t) = \mathbf{0}) = \mP(A_{ij}(t) = 1) = \rho/2+ \frac{\kappa}{\sqrt{n_1n_2}}v_{i}, \quad 1\leq t \leq \Delta \\
     \mP(A_{i}(t) = \mathbf{1}) &= 1-\mP(A_{i}(t) = \mathbf{0}) = \mP(A_{ij}(t) = 1) = \rho/2, \quad \Delta+1\leq t \leq T,
\end{align*}
where $A_{i}(t)$ denotes the $i$-th row of the matrix $A(t)$, and $\mathbf{1} \in \mathbb{R}^{n_2}$, $\mathbf{0} \in \mathbb{R}^{n_2}$ denote a vector of 1's and 0's respectively. The distribution of each network at time $t$ is denoted as $P_{v,t} = \prod_{1\leq i \leq n_1}P_v^{it}$. 

Note that $\eta(P_v^T) = \Delta$, $\|A(\Delta)-A(\Delta+1)\|_\mathrm{F}^2 = \kappa^2$, and $\kappa_0^2 = \frac{1}{20\sqrt{n_1}\rho^2\Delta(e^\alpha-1)^2}$. We are constrained by $\kappa/\sqrt{n_1n_2} < \rho/2$, which is equivalent to 
$\kappa_0^2 \leq 1/4$. Therefore, for each $v$, we have $P_v^T \in \mathcal{P}$. Similar to the construction in \Cref{lemma: edgelowerbound}, we let $\tilde{P}_v^T$ be the joint distribution of a collection of independent adjacency matrices $\{A(t)\}_{t=1}^T$ that is symmetric to $P_v^T$ with respect to time point $T/2$ and has $\eta(\tilde{P}_v^T)$ = $T - \Delta$. 

Let $Z^T_v$ and $\tilde{Z}^T_v$ be the corresponding joint private distribution generated via some \emph{node} $\alpha$ LDP mechanism satisfying \eqref{eq:nodeldp}. Write $Z^T = \frac{1}{2^{n_1}}\sum_{v \in \{1,-1\}^{n_1}}Z_v $ and $\tilde{Z}^T = \frac{1}{2^{n_1}}\sum_{v \in \{1,-1\}^{n_1}}\tilde{Z}_v^T$. 

Using the same argument as in the proof of \Cref{lemma: edgelowerbound}, it is sufficient to consider 
\[
\chi^2(Z^T,Z_0^T) + 1 = \frac{1}{4^{n_1}}\sum_{u,v \in \{-1,1\}^{n_1}}  \E_{Z^T_0}\Bigg(\frac{m_v(Z_{1:n_1}(1:T)) m_u(Z_{1:n_1}(1:T))}{m_0(Z_{1:n_1}(1:T)) m_0(Z_{1:n_1}(1:T))}\Bigg),
\]
where $m_0(\cdot) = \int Q(\cdot|a_{1:n_1}(1:T))dP_0^{T}(a_{1:n_1}(1:T))$, with $P_0^{T}$ being the $T$-fold product measure of $P_{v,T}$, is the density of $Z_0^T$. We show that $\chi^2(Z^T,Z_0^T) \leq 1/8$ with our choice $\kappa^2 = \frac{\sqrt{n_1}n_2}{20(e^\alpha-1)^2\Delta}$. 

Using the same argument as in the proof of \Cref{lemma:controlchisquare}, we have 
\[
 \E_{Z^T_0}\Bigg(\frac{m_v(Z_{1:n_1}(1:T)) m_u(Z_{1:n_1}(1:T))}{m_0(Z_{1:n_1}(1:T)) m_0(Z_{1:n_1}(1:T))}\Bigg) =  \E_{Z^{\Delta}_0}\Bigg(\frac{m_v(Z_{1:n_1}(1:\Delta)) m_u(Z_{1:n_1}(1:\Delta))}{m_0(Z_{1:n_1}(1:\Delta)) m_0(Z_{1:n_1}(1:\Delta))}\Bigg). 
\]

For simplicity, we use a generic $z$ to denote the $i$-th row of the private network $z(t)$ and write $q_1(z) = q_{i}^{(t)}(z|A_i(t)  = \mathbf{1}, \rightarrow z)$ and $q_0(z) = q_{it}(z|A_i(t) = \mathbf{0}, \rightarrow z)$, where the notation $\rightarrow z$ contains all the private information that can be used to generate $z$. Following the same calculation as in the proof of \Cref{lemma:controlchisquare}, we have for any $1 \leq i \leq n_1$ and $1\leq t \leq \Delta$, 
\begin{align*}
    \E_{Z^{\Delta}_0}\Bigg(\frac{m_v(Z_{i}(t)|\rightarrow Z_{i}(t)) m_u(Z_{i}(t)|\rightarrow Z_{i}(t))}{m_0(Z_{i}(t)|\rightarrow Z_{i}(t)) m_0(Z_{i}(t)|\rightarrow Z_{i}(t))}&\bigg| \rightarrow Z_{i}(t) \Bigg) \\ &= 1+ \frac{\kappa^2}{n_1n_2} u_iv_i\int \frac{(q_0(z))^2(q_1(z)/q_0(z)-1)^2}{(q_1(z)-q_0(z))\rho+q_0(z)}dz \\
    &\leq 1+ \frac{2\kappa^2}{n_1n_2}u_iv_i\int q_0(z)(q_1(z)/q_0(z)-1)^2dz \\
    &\leq 1+\frac{2\kappa^2}{n_1n_2}u_iv_i(e^\alpha-1)^2 \\
    &\leq \exp\Big(\frac{2\kappa^2}{n_1n_2}u_iv_i(e^\alpha-1)^2\Big)
\end{align*}
where the first inequality is due to \Cref{lemma: loweredge}. Therefore, we have

\[
\chi^2(Z^T,Z_0^T) + 1 \leq  \frac{1}{4^{n_1}}\sum_{u,v \in \{-1,1\}^{n_1}} \prod_{i=1}^{n_1} \prod_{t = 1}^\Delta \exp\Big(\frac{2\kappa^2}{n_1n_2}u_iv_i(e^\alpha-1)^2\Big). 
\]

Next, writing $U \in \mathbb{R}^{n_1}$ as a random vector with independent Rademacher entries and $\epsilon = \mathbf{1}^{\top}U/n_1$, we have 
\begin{align*}
    &\chi^2(Z^T,Z_0^T) + 1 \\
   & \leq \E_U\Bigg[\exp\Big(\frac{2\Delta\kappa^2(e^\alpha-1)^2}{n_1n_2}(\mathbf{1}^{\top}U)\Big)\Bigg] \\
     &= \int_{0}^\infty\mP\Big(\exp(2\Delta\kappa^2(e^\alpha-1)^2n_2^{-1}\epsilon) \geq u\Big) du \\
    & \leq 1+ \int_1^e\mP\Bigg(\epsilon \geq \frac{n_2\log(u)}{2\kappa^2\Delta(e^\alpha-1)^2}\Bigg) du+\int_e^\infty\mP\Bigg(\epsilon \geq \frac{n_2\log(u)}{2\kappa^2\Delta(e^\alpha-1)^2}\Bigg) du \\
    &\leq 1+\int_1^e\exp\Big(- (\log(u))^2\frac{n_1n_2^2}{(2\kappa^2\Delta(e^\alpha-1)^2)^2}\Big)+\int_e^\infty  \exp\Big(- \log(u)\frac{n_1n_2^2}{(2\kappa^2\Delta(e^\alpha-1)^2)^2}\Big)du, 
\end{align*}
where the last inequality is Hoeffding's inequality. Writing $x = \frac{n_1n_2^2}{(2\kappa^2\Delta(e^\alpha-1)^2)^2}$, we have for any $x>1$
\[
\chi^2(Z^T,Z_0^T) \leq \int_1^e\exp(-x\log^2(u)) du - \frac{1}{1-x}
\]
With the choice $x \geq 90$, it holds that 
\[
\chi^2(Z^T,Z_0^T) \leq 0.1+0.012 \leq 0.125.
\]
Therefore, it is sufficient to take 
\[
\frac{\sqrt{n_1}n_2}{\kappa^2\Delta(e^\alpha-1)^2} = 20.
\]
 to ensure $\chi^2(Z^T,Z_0^T) \leq 1/8$, which completes the proof. 

\end{proof}

\section[]{Proof of results in \Cref{sec: consistentlocalisation}}
\begin{proof}[Proof of \Cref{prop: edge upperbound}]
We write $q = (1+e^\alpha)^{-1}$, the corruption probability that $A_{ij}'(t) \neq A_{ij}(t)$ in \eqref{eq-sbc}. 
The proof relies on the observation that if \(X \sim \text{Bernoulli}(\theta)\) then the privatised \(Z\) obtained by \eqref{eq-sbc} is distributed as \(\text{Bernoulli}(q * \theta)\) where \(q * \theta := q(1 - \theta) + (1 - q)\theta = q+(1-2q)\theta\). This implies if \(X\) is the adjacency matrix of an inhomogeneous Bernoulli network and \(Z\) is a corresponding private view generated by \eqref{eq-sbc} with corruption probability \(q\), then \(Z\) is distributed as an inhomogeneous Bernoulli model with parameter matrix \(q * \Theta\), where $(q * \Theta)_{ij} = q* \theta_{ij}$. In addition, the change point structure is preserved after the privatisation but with 
\[
\min_{k=1,\dotsc,K+1}\|q*\Theta(\eta_k) - q*\Theta(\eta_k-1)\|_\mathrm{F} = (1-2q)\|\Theta(\eta_k)-\Theta(\eta_k-1)\|_\mathrm{F}. 
\]
Also, since $q*\theta$ is monotonic increasing in $\theta$ (for $q < 1/2$), we have $\rho':=\|q*\Theta(t)\|_\infty = q+(1-2q)\rho$ for any $t = 1,\dotsc,T$. Lastly, we have $\rho' \geq q \geq (1+e)^{-1}$, where the second inequality holds when $\alpha \leq 1$, and $(1+e)^{-1} \geq \log(n)/n$ for any $n>1$, which guarantees the sparsity assumption in \cite{wang2021optimal} is satisfied for the privatised inhomogeneous Bernoulli network $(A'(1), \dotsc,A'(T))$. 

The result now follows by a direct application of Theorem 1 in \cite{wang2021optimal} but with some different model parameters representing the effects of privatisation, i.e. $(K' = K, \Delta' = \Delta, \kappa_{0}', n' = n, \rho')$, where $\kappa_{0}' = \frac{(1-2q)\rho}{\rho'} \kappa_{0}$. Using the transformed parameters, the Assumption 2 in \cite{wang2021optimal} becomes 
\begin{equation*}
\label{eq:18}
(1-2q) \kappa_{0} \sqrt{\frac{\rho}{1-2q + q/\rho}} \ge C_{\alpha} \sqrt{\frac{1}{n \Delta}} \log^{1+\xi}(T),
\end{equation*}
and the localisation rate in Theorem 1 in \cite{wang2021optimal} becomes 
\begin{equation*}
\label{eq:19}
\epsilon = C_{1} \log(T) \left( \frac{\sqrt{\Delta}}{(1-2q) \kappa_{0} n \rho} + \frac{(1-2q + q/\rho)\sqrt{\log(T)}}{(1-2q)^{2 }\kappa_{0}^{2} n \rho} \right). 
\end{equation*}
Substituting $q = (1+e^\alpha)^{-1}$ yields the claimed result (with different constants). 
\end{proof}

\begin{algorithm}
    \begin{algorithmic}
        \INPUT{a vector $V \in \mathbb{R}^{d}$ with $\|V\|_{\infty} \leq 1$, privacy parameter $\alpha$}
            \State 1. Generate $\widetilde{V}$ with independent coordinates according to
    \[
        \mathbb{P}(\tilde{V}_j = 1|V_j) = 1 - \mathbb{P}(\tilde{V}_j = -1|V_{j}) = \frac{1}{2}+\frac{V_j}{2}.
    \]
    \State 2. Let $T \sim \mathrm{Ber}(e^\alpha/(e^\alpha+1)$) be independent of $\widetilde{V}$.  Generate $Z$ according to 
    \begin{equation*}
        Z \sim \begin{cases}
            \mathrm{Uniform}\Big(z \in \{-B,B\}^d \Big|\langle z,\tilde{V} \rangle \geq 0\Big), & T = 1, \\
            \mathrm{Uniform}\Big(z \in \{-B,B\}^d \Big|\langle z,\tilde{V} \rangle \leq 0\Big), & T = 0,
    \end{cases}
    \end{equation*}
    where 
    \[
        B = C_d \frac{e^\alpha+1}{e^\alpha-1} \quad \text{and} \quad C_d^{-1} =\begin{cases}  \frac{1}{2^{d-1}} \binom{d-1}{(d-1)/2}, & d \quad \text{odd}, \\
        \frac{1}{2^{d-1}+ \frac{1}{2}\binom{d}{d/2}} \binom{d-1}{d/2}, & d \quad \text{even}. 
        \end{cases}
    \]
        \caption{Privacy mechanism for $\ell_\infty$-ball with radius $1$ \citep{duchi2013local,duchi2018minimax}}
        \label{alg: samplingduchi}
        \OUTPUT{The privatised vector $Z$.}
    \end{algorithmic}
\end{algorithm}

\begin{proof}[Proof of \Cref{lemma: covariance}]
We first restate the algorithm in \Cref{alg: samplingduchi}. For simplicity, we write $\Sigma$ for $\Sigma_Z$. We start with proving \eqref{eq: operatornorm}. Given \eqref{eq:cov1} and \eqref{eq:cov2}, we have when $d$ is odd,
\[
\|\Sigma\| = \|B^2I - \E (V)(\E (V))^{\top}\|\leq B^2+\|\E (V)\|_2^2,
\]
and similarly when $d$ is even, we have 
\[
\|\Sigma\| \leq B^2+\|\E (V)\|_2^2+\frac{C_1\sqrt{d}}{\alpha^2}\sqrt{\max_{i,j}\E(V_iV_j)}. 
\]
To see \eqref{eq:cov1}, simply note that 
\[
\mathrm{Var}(Z_k) = \E Z_k^2 - (\E Z_k)^2 = B^2 -(\E V_k)^2
\]
since $\E Z_k = \E V_k$ for any $k = 1,\dotsc,d$. 

To prove \eqref{eq:cov2}, we start with for any $i,j = 1,\dotsc,d$ with $i \neq j$
\begin{equation}\label{eq:covdecomp}
    \mathrm{Cov}(Z_i, Z_j) = \mathbb{E}[\mathrm{Cov}(Z_i,Z_j|V)] + \mathrm{Cov}(\E(Z_i|V),\E(Z_j|V)) = \mathbb{E}[\mathrm{Cov}(Z_i,Z_j|V)] + \mathrm{Cov}(V_i,V_j),
\end{equation}
where we use the unbiased property $\E(Z_i|V) = V_i$ for any $i = 1,\dotsc,d$ \cite[see Appendix I.3 in][]{duchi2018minimax}. Note that
\begin{equation}\label{eq: condicov}
    \mathrm{Cov}(Z_i,Z_j|V = v) = \E(Z_iZ_j|V = v) - v_iv_j = \sum_{\tilde{v} \in \{-1,1\}^d}\E(Z_iZ_j|\tilde{v})\mP(\tilde{v}|v)- v_iv_j 
\end{equation}
and 
\begin{equation}\label{eq:condiprod}
    \E(Z_iZ_j|\tilde{v})=\pi_\alpha \E(Z_iZ_j|\langle z,\tilde{v} \rangle \geq 0)+(1-\pi_\alpha)\E(Z_iZ_j|\langle z,\tilde{v} \rangle \leq 0). 
\end{equation}

Therefore, we need to compute $\E(Z_iZ_j|\langle z,\tilde{v} \rangle \geq 0)$ and $\E(Z_iZ_j|\langle z,\tilde{v} \rangle \leq 0)$. We consider the cases of $d$ being odd and even separately below. 

\paragraph{When $d$ is odd:} we have 
\begin{align*}
    \sum_{z:\langle z,\tilde{v} \rangle \geq 0 } z_iz_j &= \sum_{l = 0}^{(d-1)/2}B^2\tilde{v}_i\tilde{v}_j\left(\binom{d-2}{l}+\binom{d-2}{l-2}-2\binom{d-2}{l-1}\right) \\
    & = B^2\tilde{v}_i\tilde{v}_j\left(\sum_{l=0}^{(d-1)/2}\binom{d-2}{l}+\sum_{l=0}^{(d-5)/2}\binom{d-2}{l}-2\sum_{l=0}^{(d-3)/2}\binom{d-2}{l}\right) \\
    & = B^2\tilde{v}_i\tilde{v}_j \left(\binom{d-2}{(d-1)/2} - \binom{d-2}{(d-3)/2}\right) \\
    & = 0
\end{align*}

Therefore $\E(Z_iZ_j|\langle z,\tilde{v} \rangle \geq 0) = 0$ and by symmetry $\E(Z_iZ_j|\langle z,\tilde{v} \rangle \leq 0) = 0.$ Hence,
$\mathrm{Cov}(Z_i,Z_j) = -\E(V_iV_j) + \mathrm{Cov}(V_i,V_j) = -\E(V_i)\E(V_j)$ when $d$ is odd. 

\paragraph{When $d$ is even:}  We have 
\begin{align*}
    \sum_{z:\langle z,\tilde{v} \rangle \geq 0 } z_iz_j &= \sum_{l = 0}^{d/2}B^2\tilde{v}_i\tilde{v}_j\left(\binom{d-2}{l}+\binom{d-2}{l-2}-2\binom{d-2}{l-1}\right) \\
    & = B^2\tilde{v}_i\tilde{v}_j\left(\sum_{l=0}^{d/2}\binom{d-2}{l}+\sum_{l=0}^{d/2-2}\binom{d-2}{l}-2\sum_{l=0}^{d/2-1}\binom{d-2}{l}\right) \\
    & = B^2\tilde{v}_i\tilde{v}_j \left(\binom{d-2}{d/2} - \binom{d-2}{d/2-1}\right) \\
    & = -\frac{2B^2}{d} \tilde{v}_i\tilde{v}_j \binom{d-2}{d/2-1}
\end{align*}
where the last equality is due to $k\binom{n}{k} = (n-k+1)\binom{n}{k-1}$. Since the set $\{z \in \{-1,1\}^d|\langle z,\tilde{v} \rangle \geq 0\}$ has cardinality $M = 2^{d-1}+\frac{1}{2}\binom{d}{d/2}$, we have 
\[
\E(Z_iZ_j|\langle z,\tilde{v} \rangle \geq 0) = -\frac{2B^2}{dM} \tilde{v}_i\tilde{v}_j \binom{d-2}{d/2-1}. 
\]
When $\langle z,\tilde{v} \rangle \leq 0$, we obtain the \emph{same} result 
\begin{align*}
    \sum_{z:\langle z,\tilde{v} \rangle \leq 0 } z_iz_j &= \sum_{l = d/2}^{d}B^2\tilde{v}_i\tilde{v}_j\left(\binom{d-2}{l}+\binom{d-2}{l-2}-2\binom{d-2}{l-1}\right) \\
    & = B^2\tilde{v}_i\tilde{v}_j\left(\sum_{l=d/2}^{d}\binom{d-2}{l}+\sum_{l=d/2-2}^{d}\binom{d-2}{l}-2\sum_{l=d/2-1}^{d}\binom{d-2}{l}\right) \\
    & = B^2\tilde{v}_i\tilde{v}_j \left(\binom{d-2}{d/2} - \binom{d-2}{d/2-1}\right) \\
    & = -\frac{2B^2}{d} \tilde{v}_i\tilde{v}_j \binom{d-2}{d/2-1}. 
\end{align*}
Then, from \eqref{eq:condiprod} we get 
\[
\E(Z_iZ_j|\tilde{v}) = \E(Z_iZ_j|\langle z,\tilde{v} \rangle \geq 0) = -\frac{2B^2}{dM} \tilde{v}_i\tilde{v}_j \binom{d-2}{d/2-1}.
\]
Now we look at 
\begin{align*}
    \frac{2B^2}{dM}\binom{d-2}{d/2-1} &= \left(\frac{e^\alpha+1}{e^\alpha-1}\right)^2 \left(\frac{2^{d-1}+ \frac{1}{2}\binom{d}{d/2}}{\binom{d-1}{d/2}} \right)^2\frac{\binom{d-2}{d/2-1}}{d(2^{d-1}+ \frac{1}{2}\binom{d}{d/2})} \\
    & = \frac{1}{d}\left(\frac{e^\alpha+1}{e^\alpha-1}\right)^2 \left(\frac{2^{d-1}+ \frac{1}{2}\binom{d}{d/2}}{\frac{d-1}{d/2}\binom{d-2}{d/2-1}}\right) \\
    & = \frac{c_\alpha}{d\alpha^2}\frac{2^{d-1}+c_d2^{d-1}/\sqrt{d}}{\frac{d-1}{d/2} c_{d-2} 2^{d-2}/\sqrt{d-2}} \\
    & = \frac{C_{d,\alpha}}{\sqrt{d}\alpha^2}
\end{align*}
where we use Stirling approximation to obtain $\binom{d}{d/2} = c_d 2^d/\sqrt{d}$ with $ \exp(-1/6)(\sqrt{2\pi})^{-1}<c_d < (\sqrt{2\pi})^{-1}$ and $1/4<c_\alpha < (e+1)^2$. Indeed, using the non-asymptotic inequalities for any even $d \geq 2$
\begin{equation}\label{eq: stirling}
    \sqrt{2\pi d}(d/e)^d \exp\Big(\frac{1}{12d+1}\Big) < d! < \sqrt{2\pi d}(d/e)^d \exp\Big(\frac{1}{12d}\Big)
\end{equation}
we have 
\[
\binom{d}{d/2} = \frac{d!}{((d/2)!)^2} \leq \frac{\exp\Big(\frac{1}{12d} - \frac{2}{6d+1}\Big)}{\sqrt{2\pi}\sqrt{d}2^{-d}} \leq \frac{2^d}{\sqrt{2\pi d}}
\]
and similarly
\[
\binom{d}{d/2} = \frac{d!}{((d/2)!)^2} \geq \frac{\exp\Big(\frac{1}{12d+1} - \frac{1}{3d}\Big)}{\sqrt{2\pi}\sqrt{d}2^{-d}} \geq \frac{\exp(-1/6)2^d}{\sqrt{2\pi d}}.
\]
Therefore, from \eqref{eq: condicov} we have that there exists $C_0 <C_{d,\alpha}< C_1$ with $C_0,C_1$ being absolute constants such that 
\[
\mathrm{Cov}(Z_i,Z_j|V = v) = -\frac{C_{d,\alpha}}{\sqrt{d}\alpha^2} \E(\tilde{V}_i\tilde{V}_j|v ) - v_iv_j = -\left(1+\frac{C_{d,\alpha}}{\sqrt{d}\alpha^2}\right)v_iv_j.
\]
It follows from \eqref{eq:covdecomp} that when $d$ is even 
\[
\mathrm{Cov}(Z_i,Z_j) = -\left(1+\frac{C_{d,\alpha}}{\sqrt{d}\alpha^2}\right)\E(V_iV_j) + \mathrm{Cov}(V_i,V_j) = -\frac{C_{d,\alpha}}{\sqrt{d}\alpha^2}\E(V_iV_j) - \E(V_i)\E(V_j). 
\]
\end{proof}

\begin{proof}[Proof of \Cref{thm: nodeupperbound}] 
First, we set
\[
\kappa^2 := \min_{k = 1,\dotsc,K}\|\Theta(\eta_k)-\Theta(\eta_{k}-1)\|_\mathrm{F}^2 = \kappa_0^2 n_1n_2\rho^2,
\]
to be the unnormalised minimal jump size in Frobenius norm and 
\begin{equation}\label{eq: localisation error}
    \varepsilon = c_5 \log(Tn_1n_2)\Bigg(\frac{\sqrt{\Delta }}{\kappa_0\rho\alpha}\sqrt{\frac{n_2}{n_1}}+\frac{{\log(Tn_1n_2)}}{\rho^2\alpha^2\kappa_0^2}\max\Bigg\{\sqrt{\frac{n_2}{n_1}},\frac{n_2}{n_1}\Bigg\}\Bigg)
\end{equation}
to be the claimed localisation error in \Cref{thm: nodeupperbound}. 

In the proof we use the notation $\kappa^2$ and it translates directly to the signal to noise condition \eqref{eq: node_stn} in terms of $\kappa_0^2$. We also use $U(t)$ and $V(t)$ to denote the privatised matrices obtained by \eqref{eq:duchisampling}, which is consistent with the notations in \Cref{algorithm:MWBS} and results in \Cref{sec: probability bounds}. 

We consider two events. The first event guarantees the quality of the randomly generated intervals. Let $\{\alpha_m\}_{m=1}^M$ and $\{\beta_m\}_{m=1}^M$ be two independent sequences selected uniformly randomly from $\{1,\dotsc,T\}$. 
\[
\mathcal{M} = \bigcap_{k = 1}^K\{\alpha \in S_k, \beta_m \in E_k, \text{for some} \,m \in \{1,\dotsc,M\}\},
\]
where $S_k = [\eta_k - 3\Delta/4, \eta_k-\Delta/2]$ and $E_k = [\eta_k+\Delta/2,\eta_k+3\Delta/4]$, $k = 1,\dotsc,K$. It is shown in \citet[][Lemma 24]{wang2021optimal} that 
\[
\mP(\mathcal{M}) \geq 1-\exp\Bigg(\log(T/\Delta)-\frac{M\Delta^2}{16T^2}\Bigg). 
\]
Next, for $0 \leq s < t < e \leq T$, consider the events
\[
\begin{split}
\mathcal{A}(s,t,e) &= \Bigg\{\Bigg|(\tilde{U}^{(s,e)}(t),\tilde{V}^{(s,e)}(t) ) - \|\tilde{\Theta}^{(s,e)}(t)\|^2_{\mathrm{F}}\Bigg| \\ &\leq C_\beta B\log(Tn_1n_2)\Big(\sqrt{n_2}\|\tilde{\Theta}^{(s,e)}(t)\|_{\mathrm{F}}+B\log(Tn_1n_2)\max\{\sqrt{n_1n_2},n_2\}\Big)\Bigg\}.
\end{split}
\]
Choosing $c,c'$ and $c'' > 3$ in \Cref{lemma: prob_bipartite3}, we have
\[
\mP(\mathcal{A}) = \mP\Bigg(\bigcup_{\substack{1\leq s<t<e\leq T}}\mathcal{A}(s,t,e)\Bigg) \geq 1- (T^{3-c_4}+T^{3-c_5}+2T^{3-c_6}) 
\]
by a union bound. The rest of the proof is conditional on the event $\mathcal{A}\cap\mathcal{M}$ and does not involve further probabilistic arguments. 

Our proof follows the standard induction-like argument for proving consistency of change point estimators \citep{wang2021optimal,fryzlewicz2014wild,wang2018high}. In particular, since the effects of node LDP is fully represented in the probabilistic arguments of analysing event $\mathcal{A}$ (c.f.\ \Cref{lemma: prob_bipartite3} and Lemma 6 in \citet{wang2021optimal}), the rest of the analytic arguments in the proof of Theorem 1 in \citet{wang2021optimal} can be applied to our problem directly. Therefore, we only point out the differences in each step between their proof and ours caused by the different concentration behaviour of $(\tilde{U}^{(s,e)}(t),\tilde{V}^{(s,e)}(t))$. To that end, we consider a generic time interval $(s,e) \subset (0,T)$ that satisfies 
\[
\eta_{r-1} \leq s \leq \eta_r \leq \dotsc \leq \eta_{r+1} \leq e \leq \eta_{r+q+1}, \quad q \geq -1
\]
and 
\[
\max\{\min\{\eta_r-s,s-\eta_{r-1}\}, \min\{\eta_{r+q+1}-e,e-\eta_{r+q}\}\} \leq \varepsilon,
\]
where $q = -1$ means that there is no change point contained in $(s,e)$ and $\varepsilon$ is given in \eqref{eq: localisation error}. A change point $\eta_p$ in $[s,e]$ is referred to as undetected if $\min\{\eta_p-s,\eta_p-e\}\geq 3\Delta/4$. Let $s_m,e_m,a_m,b_m,\tau$ and $m^*$ be defined as in the algorithm. The next four steps parallel the four steps in the proof of Theorem 1 in \cite{wang2021optimal} and establish that our algorithm 
\begin{enumerate}
    \item rejects the existence of undetected change points if $(s,e)$ does not contain any undetected change points
    \item output an estimate $b$ such that $|\eta_p-b|\leq \varepsilon$ if these is at least one undetected change point in $(s,e)$. 
\end{enumerate}

\textbf{Step 1}. Suppose that there do not exist any undetected change points within $(s,e)$. We have with 
\[
\tau > c_1 n_2\alpha^{-2} \log^2(Tn_1n_2)\max\{\sqrt{n_1n_2},n_2\},
\]
the algorithm will always correctly reject the existence of undetected change points. 

\textbf{Step 2}. Suppose that there exists an undetected change point $\eta_p \in (s,e)$. On the event $\mathcal{M}$, there exists an interval $[s_m,e_m]$ such that 
\[
\eta_p-3\Delta/4 \leq s_m \leq \eta_p-\Delta/8 \quad \text{and} \quad \eta_p+\Delta/8\leq e_m \leq \eta_p+3\Delta/4. 
\]
Now, on event $\mathcal{A}$, we have 
\[
\begin{split}
(\tilde{U}^{(s_m,e_m)}(\eta_p),&\tilde{V}^{(s_m,e_m)}(\eta_p) ) \geq \|\tilde{\Theta}^{(s_m,e_m)}(\eta_p)\|^2_{\mathrm{F}}\\
&- C_\beta B\log(Tn_1n_2)\Big(\sqrt{n_2}\|\tilde{\Theta}^{(s,e)}(\eta_p)\|_{\mathrm{F}}+B\log(Tn_1n_2)\max\{\sqrt{n_1n_2},n_2\}\Big)
\end{split}
\]
It follows from \citet[][Lemma 17]{wang2021optimal} that 
\[
\|\tilde{\Theta}^{(s_m,e_m)}(\eta_p)\|_{\mathrm{F}}^2 \geq \kappa^2\Delta/8. 
\]
Then using \eqref{eq: node_stn} we have 
\begin{equation}\label{eq:s2-1}
    \kappa^2\Delta/16 \geq \frac{c_0}{16} \frac{n_2}{\alpha^2}\log^{2+\xi}(Tn_1n_2)\max\{\sqrt{n_1n_2},n_2\} \geq  C_\beta B^2\log^2(Tn_1n_2)\max\{\sqrt{n_1n_2},n_2\}
\end{equation}
if $c_0\log^{\xi}(Tn_1n_2)> 16 C_\beta$. Also,
\begin{multline}\label{eq:s2-2}
       \frac{\|\tilde{\Theta}^{(s_m,e_m)}(\eta_p)\|_{\mathrm{F}}}{2} \geq \frac{\kappa\sqrt{\Delta}}{2\sqrt{2}}  \geq \frac{\sqrt{c_0}}{2\sqrt{2}} B\log^{1+\xi/2}(Tn_1n_2)(\max\{\sqrt{n_1n_2},n_2\})^{1/2} \geq \\
    C_\beta B \sqrt{n_2}\log(Tn_1n_2)
\end{multline}
provided $c_0\log^{\xi}(Tn_1n_2)>8C_\beta^2$. Therefore, we have for $c_0$ large enough, there exists some absolute constant $c_2$ such that 
\[
(\tilde{U}^{(s_m,e_m)}(\eta_p),\tilde{V}^{(s_m,e_m)}(\eta_p) ) \geq c_2\kappa^2\Delta. 
\]
By the definition of $m^*$, we have 
\begin{equation}\label{eq:s2-3}
    (\tilde{U}^{(s_{m*},e_{m*})}(b_{m^*}),\tilde{V}^{(s_{m*},e_{m*})}(b_{m*}) ) \geq c_2\kappa^2\Delta.
\end{equation}
Thus, with $\tau < c_2\kappa^2\Delta$, our algorithm can consistently detect the existence of undetected change points. 

\textbf{Step 3}. Suppose that there exists at least one undetected change point $\eta_p \in (s,e)$. We show that the selected interval $(s_{m^*},e_{m^*})$ indeed contains an undetected change point $\eta_p$. Suppose that 
\begin{equation}\label{eq: s3}
    \max_{s_{m*}<t<e_{m*}}\|\tilde{\Theta}^{(s_{m*},e_{m*})}(t)\|^2_{\mathrm{F}} < c_2\kappa^2\Delta/2
\end{equation}
Then 
\[
\begin{split}
&\max_{s_{m*}<t<e_{m*}}(\tilde{U}^{(s_{m*},e_{m*})}(t),\tilde{V}^{(s_{m*},e_{m*})}(t) ) \\ &\le \max_{s_{m*}<t<e_{m*}}\|\tilde{\Theta}^{(s_{m*},e_{m*})}(t)\|^2_{\mathrm{F}} + C_\beta B\log(Tn_1n_2)\Big(\sqrt{n_2}\max_{s_{m*}<t<e_{m*}}\|\tilde{\Theta}^{(s_{m*},e_{m*})}(t)\|_{\mathrm{F}}\\&+B\log(Tn_1n_2)\max\{\sqrt{n_1n_2},n_2\}\Big) \\
& \leq c_2\kappa^2\Delta/2+\sqrt{c_2/2}\kappa\sqrt{\Delta}C_\beta B\log(Tn_1n_2)\sqrt{n_2}+C_\beta B^2\log^2(Tn_1n_2)\max\{\sqrt{n_1n_2},n_2\} \\
& \leq c_2\kappa^2\Delta
\end{split}
\]
where the first inequality is due to the definition of event $\mathcal{A}$, the second inequality is due to \eqref{eq: s3}, and the last inequality is due to \eqref{eq: node_stn} with sufficiently large $c_0$. This is a contradiction to \eqref{eq:s2-3}, and therefore 
\begin{equation}\label{eq:s3-2}
    \max_{s_{m*}<t<e_{m*}}\|\tilde{\Theta}^{(s_{m*},e_{m*})}(t)\|^2_{\mathrm{F}} > c_2\kappa^2\Delta/2
\end{equation}
Then, we can conclude that $[s_{m*},e_{m*}]$ contains at least one undetected change point using the same argument as that in Step 3 in \cite{wang2021optimal}. 

\textbf{Step 4}. Continue from Step 3, we will show that 
\[
|b_{m*}-\eta_p| \leq \varepsilon,
\]
by applying Lemma 7 in \cite{wang2021optimal}. The conditions of Lemma 7 can be easily checked by letting 
\begin{align*}
    \lambda &= \max_{s_{m*}<t<e_{m*}}|(\tilde{U}^{(s_{m*},e_{m*})}(t),\tilde{V}^{(s_{m*},e_{m*})}(t) )-\|\tilde{\Theta}^{s_{m*},e_{m*}}(t)\|^2_{\mathrm{F}}| \\
    & \leq C_\beta B\log(Tn_1n_2)\Big(\sqrt{n_2}\max_{s_{m*}<t<e_{m*}}\|\tilde{\Theta}^{(s_{m*},e_{m*})}(t)\|_{\mathrm{F}}+B\log(Tn_1n_2)\max\{\sqrt{n_1n_2},n_2\}\Big) \\
    & \leq c_3 \max_{s_{m*}<t<e_{m*}}\|\tilde{\Theta}^{(s_{m*},e_{m*})}(t)\|_{\mathrm{F}}^2
\end{align*}
where the first inequality is due to the definition of event $\mathcal{A}$ and the second inequality is obtained by combining \eqref{eq:s3-2}, \eqref{eq:s2-1} and \eqref{eq:s2-2}. Then their Lemma 7 guarantees that there exists an undetected change point $\eta_p$ within $[s,e]$ with
\[
|\eta_p - b| \leq \frac{C_3\Delta\lambda}{\|\tilde{\Theta}^{s_{m*},e_{m*}}(\eta_p)\|^2_{\mathrm{F}}} \quad \text{and} \quad \|\tilde{\Theta}^{s_{m*},e_{m*}}(\eta_p)\|^2_{\mathrm{F}}\geq c'\max_{s_{m*}<t<e_{m*}}\|\tilde{\Theta}^{(s_{m*},e_{m*})}(t)\|_{\mathrm{F}}^2.
\]
Combing with \eqref{eq:s3-2}, we have 
\begin{align*}
    &|\eta_p-b| \\
    &\leq\frac{C_3\Delta C_\beta B\log(Tn_1n_2)\Big(\sqrt{n_2}\max_{s_{m*}<t<e_{m*}}\|\tilde{\Theta}^{(s_{m*},e_{m*})}(t)\|_{\mathrm{F}}+B\log(Tn_1n_2)\max\{\sqrt{n_1n_2},n_2\}\Big)}{c'\max_{s_{m*}<t<e_{m*}}\|\tilde{\Theta}^{(s_{m*},e_{m*})}(t)\|_{\mathrm{F}}^2}\\
    & = c_7 B\log(Tn_1n_2)\Bigg(\frac{\sqrt{\Delta n_2}}{\kappa}+\frac{B\max\{\sqrt{n_1n_2},n_2\}{\log(Tn_1n_2)}}{\kappa^2}\Bigg) \\ &=  c_8\log(Tn_1n_2)\Bigg(\frac{\sqrt{\Delta }n_2}{\kappa\alpha}+\frac{\max\{\sqrt{n_1n_2},n_2\}n_2{\log(Tn_1n_2)}}{\alpha^2\kappa^2}\Bigg) \\
    &=\varepsilon,
\end{align*}
which completes the proof. 
\end{proof}

\subsection{Probability bounds}\label{sec: probability bounds}
In this section, we derive necessary probability bounds for bipartite node privacy. Recall that $X_i$ denotes the $i$th \emph{row} of some general matrix $X$, $X^{\top}$ denotes the transpose of $X$, and $\|X\|$ denotes the operator norm of $X$. In particular we consider two independent copies $\{X(t)\}_{t=1}^T$ and $\{Y(t)\}_{t=1}^T$ satisfying \Cref{asp:model-bi}. Let $\{U(t)\}_{t=1}^T$ and $\{V(t)\}_{t=1}^T$ be their private versions obtained by applying the sampling mechanism \eqref{eq:duchisampling} to $\{X_i(t)\}_{t = 1}^T$ and $\{Y_i(t)\}_{t = 1}^T$ respectively. Note that 
\[
\E(U(t)) = \E(V(t)) = \E(X(t)) = \E(Y(t))=\Theta(t)
\]
since $\{U(t)\}_{t=1}^T$ and $\{V(t)\}_{t=1}^T$ are also independent copies and the sampling mechanism is unbiased.  We write 
\[
\tilde{U} = \sum_{t= 1}^T w_t U(t), \quad \tilde{V} = \sum_{t = 1}^Tw_t V(t) \quad \text{and} \quad \tilde{\Theta} =  \sum_{t = 1}^Tw_t \Theta(t)  
\]
with 
\[
\sum_{t = 1}^T w_t^2 = 1.
\]
Also we write $\Sigma_i(t)$ as the covariance matrix for $U_i(t)$ and $V_i(t)$ and applying \eqref{eq: operatornorm} with $d=n_2$ yields that when $\alpha<1$ and $n_2$ is odd,
\[
\|\Sigma_i(t)\| \leq B^2+\|\Theta_i(t)\|_2^2\leq B^2+n_2\rho^2\leq 2B^2,
\]
for any $i = 1,\dotsc,n_1$ and $t = 1,\dotsc,T$. Similarly when $\alpha<1$ and $n_2$ is even
\[
\|\Sigma_i(t)\| \leq B^2+n_2\rho^2+\frac{c\sqrt{n_2}}{\alpha^2} \leq 3B^2.
\]
Therefore, we have 
\begin{equation}\label{eq:prob_operatornorm}
    \max_{i = 1,\dotsc,n_1,t = 1,\dotsc,T}\|\Sigma_i(t)\|\leq 3B^2
\end{equation}
for both $n_2$ is odd and even cases. 

\begin{lemma}\label{lemma: prob_bipartite1}
Let $k_i \in \mathbb{R}^{n_2}$ be an arbitrary vector. Then for any $\varepsilon>0$, we have 
\begin{multline*}
    \mP\left(\left|\sum_{t = 1}^Tw_t\sum_{ i = 1}^{n_1} k_i^{\top}(V_i(t) - \Theta_i(t))\right| > \epsilon\right) \\ \leq 2 \exp\left( \frac{-\tfrac{1}{2}\epsilon^2}{ 3B^2 \sum_{i=1}^{n_1}\|k_i\|_2^2 + \max_{\substack{i = 1,\dotsc,n_1}}\|k_i\|_2\sqrt{n_2}B2\epsilon/3} \right). 
\end{multline*}
\end{lemma}
\begin{proof}
The proof is due to an application of Bernstein's inequality \cite[][Theorem 2.8.4]{vershynin2018high}. Notice that 
\begin{align*}
    \E\left(\sum_{t = 1}^Tw_t\sum_{ i = 1}^{n_1}k_i^{\top}(V_i(t) - \Theta_i(t))\right)^2 &= \sum_{t = 1}^T w_t^2 \sum_{i = 1}^{n_1}\E[k_i^{\top}(V_i(t) - \Theta_i(t))]^2 \\
    & = \sum_{t = 1}^T w_t^2\sum_{i=1}^{n_1}\E\left[\sum_{j = 1}^{n_2}k_{ij}(V_{ij}(t) - \Theta_{ij}(t))\right]^2 \\
    & = \sum_{t = 1}^T w_t^2\sum_{i=1}^{n_1}k_i^{\top}\Sigma_i(t)k_i \\
    & \leq \sum_{i=1}^{n_1}\|k_i\|_2^2\max_{i,t}\|\Sigma_i(t)\| \\
    & \leq 3B^2 \sum_{i = 1}^{n_1}\|k_i\|_2^2
\end{align*}
where the first line is due to the independence across $t$ and $i = 1, \dotsc, n_1$, the first inequality is due to the definition of operator norm $\|\Sigma_i\|$ and $\sum_{t = 1}^Tw_t^2 = 1$, and in the last line we use \eqref{eq:prob_operatornorm}.  Also since $|k_i^{\top}(V_i(t) - \Theta_i(t))| \leq \|k_i\|_2\|V_i(t) - \Theta_i(t)\|_2 \leq 2\|k_i\|_2\sqrt{n_2}B$ and $w_t \leq 1$, we have 
\[
\mP\left(\left|\sum_{t = 1}^Tw_t\sum_{ i = 1}^{n_1} k_i^{\top}(V_i(t) - \Theta_i(t))\right| > \epsilon\right) \leq 2\exp \left( \frac{-\tfrac{1}{2}\epsilon^2}{ 3B^2 \sum_{i=1}^{n_1}\|k_i\|_2^2 + \max_{\substack{i = 1,\dotsc,n_1}}\|k_i\|_2\sqrt{n_2}B2\epsilon/3} \right)
\]
by Bernstein's inequality, as claimed. 
\end{proof}

\begin{lemma}\label{lemma: prob_bipartite2}
Let $k_i = \sum_{t = 1}^T w_t(V_i(t) - \Theta_i(t))$. Then there exist absolute constants $C, c>0$ such that 
\[
\mP\left(\max_{i = 1,\dotsc,n_1}\|k_i\|_2 \geq C\sqrt{n_2}B\log(Tn_1n_2) \right) \leq T^{-c}. 
\]
\end{lemma}
\begin{proof}
First note that $\|V_i(t) - \Theta_i(t)\|_2 \leq 2\sqrt{n_2}B$ for any $i = 1,\dotsc,n_1$ and $t = 1,\dotsc,T$. Also, we have 
\[
\E\|V_i(t) - \Theta_i(t)\|_2^2 \leq \E\|V_i(t)\|_2^2 = n_2 B^2 \quad\text{and}\quad\E (V_i(t) - \Theta_i(t))(V_i(t) - \Theta_i(t))^{\top} = \Sigma_i(t),
\]
and $\max_{\substack{i,t}}\|\Sigma_i(t)\| \leq 3B^2$. Next, we apply the matrix Bernstein inequality for rectangular matrices \cite[][Exercise 5.4.15]{vershynin2018high} to $k_i = \sum_{t = 1}^T w_t(V_i(t) - \Theta_i(t))$, for any fixed $i = 1,\dotsc,n_1$, and obtain
\[
\mP(\|k_i\|_2 \geq t)\leq 2(n_2+1)\exp\left(-\frac{t^2/2}{\sigma^2+\sqrt{n_2}Bt/3}\right)
\]
where $$\sigma^2 = \max\left(\sum_{t = 1}^Tw_t^2\E\|V_i(t) - \Theta_i(t)\|_2^2, \|\sum_{t=1}^T w_t^2\Sigma_i(t)\|\right) \leq \max(n_2B^2,3B^2)\leq 3n_2B^2.$$ 
Next, using a union bound we have 
\[
\mP\left(\max_{i = 1,\dotsc,n_1}\|k_i\|_2 \geq t\right) \leq 4n_1n_2\exp\left(-\frac{t^2/2}{3n_2B^2+\sqrt{n_2}Bt/3}\right).
\]
Choosing $t = C\sqrt{n_2}B\log(Tn_1n_2)$ for some absolute constant $C$ large enough in the above leads to 
\[
\mP\left(\max_{i = 1,\dotsc,n_1}\|k_i\|_2 \geq C\sqrt{n_2}B\log(Tn_1n_2) \right) \leq T^{-c}. 
\]

\end{proof}
\begin{lemma}\label{lemma: prob_bipartite3} There exist absolute constants $c,c',c''>0$ such that
\[
\begin{split}
    \mP&\Bigg(\Bigg|\sum_{i = 1}^{n_1}\tilde{U}_i^{\top}\tilde{V}_i - \|\tilde{\Theta}\|^2_{\mathrm{F}}\Bigg| > C_\beta B\log(Tn_1n_2)\Big(\sqrt{n_2}\|\tilde{\Theta}\|_{\mathrm{F}}+B\log(Tn_1n_2)\max\{\sqrt{n_1n_2},n_2\}\Big)\Bigg) \\ &\leq T^{-c}+T^{-c'}+2T^{-c''}.
\end{split}
\]
\end{lemma}
\begin{proof}
Note that $\sum_{i = 1}^{n_1}\tilde{U}_i'\tilde{V}_i - \|\tilde{\Theta}\|^2_{\mathrm{F}} = I+II+III$, where
\[
I = \sum_{i=1}^{n_1}(\tilde{U}_i-\tilde{\Theta}_i)^{\top}(\tilde{V}_i - \tilde{\Theta}_i), \quad II = \sum_{i=1}^{n_1}\tilde{\Theta}_i^{\top}(\tilde{V}_i-\tilde{\Theta}_i), \quad \text{and} \quad III=\sum_{i=1}^{n_1}\tilde{\Theta}_i^{\top}(\tilde{U}_i-\tilde{\Theta}_i). 
\]
It is sufficient to bound $I$ and $II$, since $II$ and $III$ are independent copies of each other. 

We start by bounding $I$ using \Cref{lemma: prob_bipartite1} and \Cref{lemma: prob_bipartite2}. Writing $k_i = \tilde{U}_i-\tilde{\Theta}_i = \sum_{t= 1}^Tw_t(U_i(t) - \Theta_i(t))$, we have from \Cref{lemma: prob_bipartite1} that conditional on $\{U(t)\}_{t= 1}^T$
\begin{align*}
    \mP_{V|U}\left(\left|I\right| > \epsilon\right) & = \mP\left(\left|\sum_{t = 1}^Tw_t\sum_{ i = 1}^{n_1} k_i^{\top}(V_i(t) - \Theta_i(t))\right| > \epsilon\right)\\ & \leq 2\exp \left( \frac{-\tfrac{1}{2}\epsilon^2}{ 3B^2 \sum_{i = 1}^{n_1} \|k_i\|_2^2 + \max_{\substack{i = 1,\dotsc,n_1}}\|k_i\|_2\sqrt{n_2}B2\epsilon/3} \right) \\
    & \leq 2\exp \left( \frac{-\tfrac{1}{2}\epsilon^2}{ 3n_1B^2 \max_{\substack{i = 1,\dotsc,n_1}}\|k_i\|_2^2 + \max_{\substack{i = 1,\dotsc,n_1}}\|k_i\|_2\sqrt{n_2}B2\epsilon/3} \right).
\end{align*}
Now by \Cref{lemma: prob_bipartite2}, we have 
\[
\mP_{U}\left(\max_{i = 1,\dotsc,n_1}\|k_i\|_2 \geq C\sqrt{n_2}B\log(Tn_1n_2) \right) \leq T^{-c}. 
\]
Therefore, for any $\varepsilon>0$, it holds that 
\[
\mP(|I|>\varepsilon) \leq 2\exp \left( \frac{-\tfrac{1}{2}\epsilon^2}{ 3Cn_1n_2B^4 \log^2(Tn_1n_2) + 2Cn_2B^2\log(Tn_1n_2)\epsilon/3} \right) + T^{-c}
\]
and there exists some constant $c',C'$ such that 
\[
\mP\left(|I|>C'B^2\log^2(Tn_1n_2)\max\{\sqrt{n_1n_2},n_2\}\right) \leq T^{-c'}+T^{-c}. 
\]
Now onto term $II$. Applying \Cref{lemma: prob_bipartite1} with $k_i = \tilde{\Theta}_i$ yields 
\[
\mP(|II|> \varepsilon) \leq 2\exp \left( \frac{-\tfrac{1}{2}\epsilon^2}{ 3B^2 \|\tilde{\Theta}\|_{\mathrm{F}} + \max_{\substack{i = 1,\dotsc,n_1}}\|\tilde{\Theta}_i\|_2\sqrt{n_2}B2\epsilon/3} \right).
\]
Therefore there exist absolute constants $c'',C''$ such that 
\[
\mP\Big(|II|> C''\sqrt{n_2}B\|\tilde{\Theta}\|_{\mathrm{F}}\log(T)\Big) \leq T^{-c''},
\]
since $\max_{\substack{i = 1,\dotsc,n_1}}\|\tilde{\Theta}_i\|_2 = \sqrt{\max_{\substack{i = 1,\dotsc,n_1}}\|\tilde{\Theta}_i\|_2^2} \leq \|\tilde{\Theta}_i\|_{\mathrm{F}}$ and the claim follows. 
\end{proof}

\end{document}